\documentclass[10pt]{amsart}
\usepackage[margin=1in,letterpaper,portrait]{geometry}
\usepackage{latexsym,amssymb,verbatim,multirow,graphicx,epsfig}
\usepackage{ytableau}

\theoremstyle{plain}
   \newtheorem{theorem}{Theorem}[section]
   \newtheorem{proposition}[theorem]{Proposition}
   \newtheorem{lemma}[theorem]{Lemma}
   \newtheorem{corollary}[theorem]{Corollary}
   \newtheorem{conjecture}[theorem]{Conjecture}
   
   \newtheorem*{theorem*}{Theorem}
\theoremstyle{definition}
   \newtheorem{definition}[theorem]{Definition}

   \newtheorem{question}[theorem]{Question}
   \newtheorem{remark}[theorem]{Remark}

\numberwithin{equation}{section}

\newcommand\qbin[3]{\left[\begin{matrix} #1 \\ #2 \end{matrix} \right]_{#3}}
\newcommand\hook[2]{{\left( #2-#1, 1^{#1} \right)}}

\newcommand\Symm{\mathfrak{S}}

\newcommand\defining{\operatorname{def}}

\newcommand\one{\mathbf{1}}

\newcommand\Hom{\operatorname{Hom}}

\newcommand\Ind{\operatorname{Ind}}
\newcommand\Par{\operatorname{Par}}

\newcommand\Irr{\operatorname{Irr}}

\newcommand\Tr{\operatorname{Tr}}
\newcommand\wt{\operatorname{wt}}
\newcommand\maj{\operatorname{maj}}
\newcommand\sgn{\operatorname{sgn}}

\newcommand\normchi{\widetilde{\chi}}
\newcommand\Class{\operatorname{Cl}}
\newcommand\Nref{\operatorname{N^{\operatorname{ref}}}}
\newcommand\Nhyp{\operatorname{N^{\operatorname{hyp}}}}
\newcommand\hcprod{{*}}

\newcommand\OOO{{\mathcal{O}}}

\newcommand\CCCC{{\mathcal{C}}}

\newcommand\CCC{{\operatorname{Cusp}}}

\newcommand\llambda{{\underline{\lambda}}}

\newcommand\CC{{\mathbb{C}}}
\newcommand\ZZ{{\mathbb{Z}}}

\newcommand\QQ{{\mathbb{Q}}}

\newcommand\FF{{\mathbb{F}}}

\newcommand\HH{{\mathbb{H}}}

\begin{document}

\title[Reflection factorizations of Singer cycles]
{Reflection factorizations of Singer cycles}

\author{J.B. Lewis}
\author{V. Reiner}
\author{D. Stanton}
\email{(jblewis,reiner,stanton)@math.umn.edu}
\address{School of Mathematics\\
University of Minnesota\\
Minneapolis, MN 55455, USA}

\thanks{Work partially supported by NSF grants DMS-1148634 and
DMS-1001933.}


\begin{abstract}
The number of shortest factorizations into reflections for a 
Singer cycle in $GL_n(\FF_q)$ is shown to be $(q^n-1)^{n-1}$.  
Formulas counting factorizations of any length, and counting those with reflections of fixed
conjugacy classes are also given.  
The method is a standard character-theory technique, requiring the compilation of irreducible character values for Singer cycles, semisimple reflections, and transvections.   The results suggest several open problems and questions, which are discussed at the end.
\end{abstract}


\maketitle


\section{Introduction and main result}

This paper was motivated by two classic results on the number
$t(n,\ell)$ of ordered factorizations $(t_1,\ldots,t_\ell)$
 of an $n$-cycle $c=t_1 t_2 \cdots t_\ell$ in the symmetric group $\Symm_n$, 
where each $t_i$ is a transposition.
\vskip.1in
\noindent
\begin{theorem*}[{Hurwitz \cite{Hurwitz}, D\'{e}nes \cite{Denes}}]
For $n \geq 1$, one has
$
\label{Denes-theorem}
t(n,n-1)=n^{n-2}.
$
\end{theorem*}
\begin{theorem*}[{Jackson \cite[p.\ 368]{Jackson}}]
For $n \geq 1$, more generally $t(n,\ell)$ has ordinary generating function
\begin{equation}
\label{Jackson-ordinary-gf-product}
\sum_{\ell \geq 0} t(n,\ell) x^\ell
=  n^{n-2} x^{n-1} \prod_{k=0}^{n-1}  \left( 1 - x n\left(\frac{n-1}{2}-k\right) \right)^{-1} 
\end{equation}
and explicit formulas
\begin{align}
\label{Jackson-difference-formula}
t(n,\ell) 
=\frac{n^\ell}{n!} \sum_{k=0}^{n-1}(-1)^k \binom{n-1}{k} 
   \left( \frac{n-1}{2}-k \right)^\ell 
=\frac{(-n)^\ell}{n!}(-1)^{n-1} \left[ \Delta^{n-1}(x^\ell)\right]_{x=\frac{1-n}{2}}.
\end{align}
\end{theorem*}
\noindent
Here the difference operator $\Delta(f)(x):=f(x+1)-f(x)$ satisfies
$
\Delta^n(f)(x):=\sum_{k=0}^n (-1)^{n-k} \binom{n}{k} f(x+k).
$

Our goals are $q$-analogues, replacing the symmetric group $\Symm_n$ with
the {\it general linear group} $GL_n(\FF_q)$, replacing transpositions with
{\it reflections}, and replacing an $n$-cycle with a {\it Singer cycle} $c$:
the image of a generator for the cyclic group
$\FF_{q^n}^\times \cong \ZZ/(q^n-1)\ZZ$ under any embedding
$
\FF_{q^n}^\times \hookrightarrow GL_{\FF_q}(\FF_{q^n}) \cong GL_n(\FF_q)
$
that comes from a choice of $\FF_q$-vector space isomorphism 
$\FF_{q^n} \cong \FF_q^n$.  The analogy between Singer cycles in $GL_n(\FF_q)$ and
$n$-cycles in $\Symm_n$ is reasonably
well-established \cite[\S 7]{StantonWebbR},  \cite[\S\S 8-9]{StantonWhiteR}.
Fixing such a Singer cycle $c$,
denote by $t_q(n,\ell)$ the number of ordered factorizations $(t_1,\ldots,t_\ell)$ 
of $c=t_1 t_2 \cdots t_\ell$ in which each $t_i$ is a {\it reflection} 
in $GL_n(\FF_q)$, that is, the fixed space $(\FF_q^n)^{t_i}$ is a hyperplane in $\FF_q^n$.

\begin{theorem}
\label{q-Denes-theorem}
For $n \geq 2$, one has
$
t_q(n,n)=(q^n-1)^{n-1}.
$
\end{theorem}

\begin{theorem}
\label{q-factorization-theorem}
For $n \geq 2$, more generally $t_q(n,\ell)$ has ordinary generating function
\begin{equation}
\label{q-Jackson-ordinary-gf-product}
\sum_{\ell \geq 0} t_q(n,\ell) x^\ell
= (q^n-1)^{n-1}  x^n \cdot
 \left(1+x[n]_q \right)^{-1} \prod_{k=0}^{n-1} \left(1+x[n]_q(1+q^k-q^{k+1})\right)^{-1}
\end{equation}
and explicit formulas
\begin{align}
 \label{q-Jackson-sum-formula}
 \qquad
t_q(n,\ell) 
&= \frac{(-[n]_q)^\ell}{q^{\binom{n}{2}}(q;q)_n} 
  \left(
    (-1)^{n-1} (q;q)_{n-1}+
    \sum_{k=0}^{n-1}(-1)^{k+n} q^{\binom{k+1}{2}} \qbin{n-1}{k}{q} 
                           (1+q^{n-k-1}-q^{n-k})^\ell \right) \\
 \label{q-Jackson-difference-formula}
 &= (1-q)^{-1}\frac{(-[n]_q)^\ell}{[n]!_q} \left[ \Delta_q^{n-1} 
        \biggl(\frac{1}{x}-\frac{(1+x(1-q))^\ell}{x}\biggr)\right]_{x=1}\\
\label{tot-num-formula}
&=
[n]_q^{\ell-1}\sum_{i=0}^{\ell-n} (-1)^i (q-1)^{\ell-i-1} \binom{\ell}{i}  
\qbin{\ell-i-1}{n-1}{q}.
\end{align}
\end{theorem}
The $q$-analogues used above and elsewhere in the paper are defined as follows:
\[
\begin{aligned}
\qbin{n}{k}{q}&:=\frac{[n]!_q}{[k]!_q [n-k]!_q}, \,\, \text{ where } \,
[n]!_q:=[1]_q [2]_q \cdots [n]_q \,\, \text{ and } \,
[n]_q:=1+q+q^2+\cdots+q^{n-1}, \\
(x;q)_n&:=(1-x)(1-xq)(1-xq^2)\cdots(1-xq^{n-1}), \,\, \textrm{ and }\\
\Delta_q(f)(x)&:=\frac{f(x)-f(qx)}{x-qx}, \text{ so that }
\end{aligned}
\]
\begin{equation}
\label{q-diff-iterate}
\Delta_q^n(f)(x)
 =\frac{1}{q^{\binom{n}{2}} x^n(1-q)^n} 
 \sum_{k=0}^n (-1)^{n-k} q^{\binom{k}{2}} \qbin{n}{k}{q}f(q^{n-k} x).
\end{equation}
The equivalence of the three formulas \eqref{q-Jackson-sum-formula},
\eqref{q-Jackson-difference-formula}, \eqref{tot-num-formula} for $t_q(n,\ell)$
is explained in Proposition~\ref{q-Jackson-sum-equals-tot-num-lemma} below.

In fact, we will prove the following refinement of Theorem~\ref{q-factorization-theorem} 
for $q>2,$ having no counterpart for $\Symm_n$.  
Transpositions are all conjugate within $\Symm_n$, but 
the conjugacy class of a 
reflection $t$ in $GL_n(\FF_q)$ for $q>2$ varies with its 
determinant $\det(t)$ in $\FF_q^\times$.  
When $\det(t)=1$, the reflection $t$ is 
called a {\it transvection} \cite[XIII \S 9]{Lang}, 
while $\det(t) \neq 1$ means that 
$t$ is a {\it semisimple reflection}.  One can associate 
to an ordered factorization $(t_1,\ldots,t_\ell)$ of $c=t_1 t_2 \cdots t_\ell$ 
the sequence of determinants 
$(\det(t_1),\ldots,\det(t_\ell))$ in $\FF_q^\ell$, having product $\det(c)$.

\begin{theorem}
\label{fixed-det-sequence-theorem}
Let $q>2.$ 
Fix a Singer cycle $c$ in $GL_n(\FF_q)$ and a sequence $\alpha=(\alpha_i)_{i=1}^\ell$ in 
$(\FF_q^\times)^\ell$ with $\prod_{i=1}^\ell \alpha_i = \det(c)$.  
Let $m$ be the number of values $i$ such that $\alpha_i = 1$.  
Then one has $m \le \ell-1,$
and the number of ordered reflection factorizations $c=t_1 \cdots t_\ell$
with $\det(t_i)=\alpha_i$ depends only upon $\ell$ and $m$.
This quantity $t_q(n,\ell,m)$ is given by these formulas:
\begin{align}
\label{det-sequence-difference-formula}
t_q(n,\ell,m) &=
[n]_q^{\ell - 1} \sum_{i = 0}^{\min(m,\ell-n)} (-1)^i \binom{m}{i} \qbin{\ell - i - 1}{n - 1}{q}
\\
\label{q-diff-nlm}
&= \frac{[n]_q^\ell}{[n]!_q} 
\left[ \Delta_q^{n-1}\bigl( (x-1)^mx^{\ell-m-1}\bigr) \right]_{x=1}.
\end{align}
In particular, setting $\ell=n$ in \eqref{det-sequence-difference-formula}, 
the number of shortest such factorizations is
\[
t_q(n,n,m)=
[n]_q^{n - 1},
\]
which depends neither on the sequence $\alpha=(\det(t_i))_{i=1}^\ell$ 
nor on the number of transvections $m$.
\end{theorem}

\noindent
The equivalence of the formulas \eqref{det-sequence-difference-formula}
and \eqref{q-diff-nlm} for $t_q(n,\ell,m)$ is also explained in
Proposition~\ref{q-Jackson-sum-equals-tot-num-lemma} below.

Theorems~\ref{q-factorization-theorem} 
and \ref{fixed-det-sequence-theorem} are proven via 
a standard character-theoretic approach.  This approach is reviewed
quickly in Section~\ref{character-approach-section}, followed
by an outline of ordinary character theory for
$GL_n(\FF_q)$ in Section~\ref{GL-character-theory}.   
Section~\ref{character-values-section} either 
reviews or derives the needed explicit character values for 
four kinds of conjugacy classes:
the identity element, Singer cycles, semisimple reflections, and transvections.
Then Section~\ref{main-result-proof-section} assembles these calculations into
the proofs of  Theorems~\ref{q-factorization-theorem} and \ref{fixed-det-sequence-theorem}.
Section~\ref{questions-remarks} closes with some further remarks and questions.

Although Theorem~\ref{fixed-det-sequence-theorem} is stated for $q>2$, something interesting
also occurs for $q=2$.
All reflections in $GL_n(\FF_2)$ are transvections, thus
one always has $m=\ell$ for $q=2$. Furthermore,  
one can see that \eqref{tot-num-formula}, \eqref{det-sequence-difference-formula}
give the same answer when both $q=2$ and $m=\ell$. This reflects a striking dichotomy in our
proofs:  for $q > 2$ the only 
contributions to the computation come from irreducible characters of
$GL_n(\FF_q)$ arising as constituents of parabolic inductions of 
characters of $GL_1(\FF_q)$,  
while for $q =2$ the cuspidal characters for $GL_s(\FF_q)$
with $s \geq 2$ play a role, miraculously giving 
the same polynomial $t_q(n,\ell)$ in $q$ evaluated at $q=2$.

\begin{question} 
Can one derive the formulas \eqref{tot-num-formula} or 
\eqref{det-sequence-difference-formula} 
via {\it inclusion-exclusion}
more directly?
\end{question}

\begin{question}
\label{q-Denes-proof-question}
Can one derive Theorem~\ref{q-Denes-theorem}
{\it bijectively}, or by an {\it overcount} in the 
spirit of  D\'enes \cite{Denes},
that counts factorizations of all conjugates of a Singer cycle, and 
then divides by the conjugacy class size?
\end{question}
%

\section{The character theory approach to factorizations}
\label{character-approach-section}

We recall the classical approach to factorization
counts, which goes back to Frobenius \cite{Frobenius}.

\begin{definition}
Given a finite group $G$, let 
$\Irr(G)$ be the set of its irreducible ordinary (finite-dimensional,
complex) representations $V$.  For each $V$ in $\Irr(G)$, denote by
$\deg(V)$ the {\it degree} $\dim_\CC V$, and
let 
$
\chi_{V}(g)=\Tr(g: V \rightarrow V)
$
be its {\it character value} at $g$,
along with  $\normchi_{V}(g):=\frac{\chi_V(g)}{\deg(V)}$ the
{\it normalized character value}. Both functions $\chi_{V}(-)$ and $\normchi_{V}(-)$ on $G$
extend by $\CC$-linearity to functionals on the {\it group algebra} $\CC G$. 
\end{definition}

\begin{proposition}[Frobenius \cite{Frobenius}]
\label{conj-count-prop}
Let $G$ be a finite group, and $A_1,\ldots,A_\ell \subset G$
unions of conjugacy classes in $G$.  Then for $g$ in $G$,
the number of ordered factorizations $(t_1,\ldots,t_\ell)$ with
$g=t_1 \cdots t_\ell$ and $t_i$ in $A_i$ for $i=1,2,\ldots,\ell$ is
\begin{equation}
\label{Chapuy-Stump-varying-class-answer}
\frac{1}{|G|} \sum_{V \in \Irr(G)} 
 \deg(V) \cdot \chi_V(g^{-1}) \cdot \normchi_{V}(z_1) \cdots \normchi_{V}(z_\ell).
\end{equation}
where $z_i:=\sum_{t \in A_i} t$ in $\CC G$.
\end{proposition}

This lemma was a main tool used by Jackson \cite[\S 2]{Jackson-older},
as well as by Chapuy and Stump \cite[\S 4]{ChapuyStump} in their solution of the analogous
question in well-generated complex reflection groups.  The proof follows from a straightforward
computation in the group algebra $\CC G$ coupled with the isomorphism of $G$-representations 
$
\CC G \cong \bigoplus_{V \in \Irr(G)} V^{\oplus \deg(V)}
$; it may be found for example in \cite[Thm.~1.1.12]{LandoZvonkin}, \cite[Thm.~2.5.9]{LuxPahlings}.


\section{Review of ordinary characters of $GL_n(\FF_q)$}
\label{GL-character-theory}

The ordinary character theory of $GL_n:=GL_n(\FF_q)$ was worked out by
Green \cite{Green}, and has been reworked many times.  
Aside from Green's paper, some useful references for us in what 
follows will be Macdonald \cite[Chaps. III, IV]{Macdonald}, and
Zelevinsky \cite[\S 11]{Zelevinsky}.

\subsection{Parabolic or Harish-Chandra induction}
The key notion is that of {\it parabolic}
or {\it Harish-Chandra induction}:  given an integer composition
$\alpha=(\alpha_1,\ldots,\alpha_m)$ of $n$, so that $\alpha_i > 0$
and $|\alpha|:=\sum_i \alpha_i=n$, and 
class functions $f_i$ on $GL_{\alpha_i}$ for $i=1,2,\ldots,m$, 
one produces a class function $f_1 \hcprod f_2 \hcprod \cdots \hcprod f_m$ on $GL_n$
defined as follows.  Regard the $m$-tuple $(f_1,\ldots,f_m)$ as a class function on the 
block upper-triangular parabolic subgroup $P_\alpha$ inside $GL_n$,
whose typical element is
\begin{equation}
\label{parabolic-format}
p=
\begin{bmatrix}
A_{1,1} & * & \cdots & *\\
0   & A_{2,2}& \cdots& *\\
\vdots& \vdots & \ddots& \vdots\\
0     & 0 &\cdots       &A_{m,m}
\end{bmatrix}
\end{equation}
with $A_{i,i}$ an invertible $\alpha_i \times \alpha_i$ matrix,
via $(f_1,\ldots,f_m)(p)=\prod_{i=1}^m f_i(A_{i,i})$.
Then apply (ordinary) induction of characters from $P_\alpha$ up to $GL_n$.
In other words, for an element $g$ in $GL_n$ one has
\begin{equation}
\label{parabolic-induction-formula}
(f_1 \hcprod f_2 \hcprod \cdots \hcprod f_m)(g):=
\frac{1}{|P_\alpha|} \sum_{\substack{h \in G:\\ hgh^{-1} \in P_\alpha}} 
  f_1(A_{1,1}) \cdots f_m(A_{m,m})
\qquad 
\text{ if }hgh^{-1}\text{ looks as in }\eqref{parabolic-format}.
\end{equation}

Identify representations $U$ up to equivalence with their
characters $\chi_U$.  The parabolic induction product $(f,g) \longmapsto f\hcprod g$
gives rise to a graded, associative product on the graded
$\CC$-vector space 
\[
\Class(GL_*) = \bigoplus_{n \geq 0} \Class(GL_n)
\]
which is the direct sum of class functions on 
all of the general linear groups, with $\Class(GL_0)=\CC$ by convention.

\subsection{Parametrizing the $GL_n$-irreducibles}
\label{irreducible-parametrization-section}

A $GL_n$-irreducible $U$ is called {\it cuspidal} if $\chi_U$ 
does not occur as a constituent in any induced
character $f_1 \hcprod f_2$ for compositions $n=\alpha_1+\alpha_2$ with 
$\alpha_1,\alpha_2 >0$. Denote by $\CCC_n$ the set of
all such cuspidal irreducibles $U$ for $GL_n$,
and say that {\it weight $\wt(U)=n$}.  Let $\Par_n$ denote the
partitions $\lambda$ of $n$ (that is, $|\lambda|:=\sum_i \lambda_i=n$),
and define 
\begin{align*}
\Par&:=\bigsqcup_{ n \geq 0} \Par_n,\\
\CCC&:=\bigsqcup_{ n \geq 1} \CCC_n
\end{align*}
the sets of {\it all} partitions, and {\it all} cuspidal 
representations for all groups $GL_n$.  Then the 
$GL_n$-irreducible characters can be indexed as
$
\Irr(GL_n)=\{ \chi^{\llambda} \}
$ 
where $\llambda$ runs through the set of all functions 
\[
 \begin{array}{rcl}
\CCC &\overset{\llambda}{\longrightarrow}& \Par \\
U & \longmapsto & \lambda(U)
\end{array}
\]
having the property that
\begin{equation}
\label{irreducible-weight-condition}
\sum_{U \in \CCC} \wt(U) \, |\lambda(U)|=n.
\end{equation}
Although $\CCC$ is infinite, 
this condition \eqref{irreducible-weight-condition}
implies that $\llambda$ can only take on finitely many
non-$\varnothing$ values $\lambda(U_1),\ldots,\lambda(U_m)$,
and in this case
\begin{equation}
\label{general-irreducibles-are-induced}
\chi^{\llambda} = \chi^{U_1,\lambda(U_1)} \hcprod \cdots \hcprod \chi^{U_m,\lambda(U_m)}
\end{equation}
where each $\chi^{U,\lambda}$ is what Green \cite[\S 7]{Green} 
called a {\it primary irreducible character}.  In particular,
a cuspidal character $U$ in $\CCC_n$ is the same as
the primary irreducible $\chi^{U,(1)}$.

\subsection{Jacobi-Trudi formulas}
\label{Jacobi-Trudi-section}

We recall from symmetric function theory 
the {\it Jacobi-Trudi} and {\it dual Jacobi-Trudi formulas} \cite[I (3.4),(3.5)]{Macdonald}
formulas.  For a 
partition  $\lambda=(\lambda_1 \geq \cdots \geq \lambda_\ell)$ 
with largest part $m:=\lambda_1$, these formulas express a {\it Schur function} $s_\lambda$
either as an integer sum of products of {\it complete homogeneous} symmetric
functions $h_n=s_{(n)}$, or of {\it elementary} symmetric functions $e_n=s_{(1^n)}$:
\[
\begin{array}{rcccl}
s_\lambda
&=&\det( h_{\lambda_i-i+j} )
&=&\displaystyle
    \sum_{w \in \Symm_\ell}  
    \sgn(w) h_{\lambda_1-1+w(1)} \cdots h_{\lambda_\ell-\ell + w(\ell)}, \\
s_\lambda 
&=&\det( e_{\lambda'_i-i+j} )
&=&\displaystyle
    \sum_{w \in \Symm_{m}}  
    \sgn(w) e_{\lambda'_1-1+w(1)} \cdots e_{\lambda'_m-m + w(m)} .
\end{array}
\]
Here $\lambda'$ is the usual {\it conjugate} or {\it transpose} partition 
to $\lambda$.  Also $h_0=e_0=1$ and $h_n=e_n=0$ if $n < 0$.

The special case of primary irreducible $GL_n$-characters $\chi^{U,(n)}, \chi^{U,(1^n)}$
corresponding to the single row partitions $(n)$ and single column partitions
$(1^n)$ are called {\it generalized trivial} and 
{\it generalized Steinberg characters}, respectively,
by Silberger and Zink \cite{SilbergerZink}.  One has
analogous formulas expressing 
any primary irreducible character $\chi^{U,\lambda}$ 
virtually in terms of parabolic induction products of such
characters:
\begin{align}
\label{primary-irreducible-Jacobi-Trudi}
\chi^{U,\lambda}&= \sum_{w \in \Symm_\ell} 
        \sgn(w) \chi^{U,(\lambda_1-1+w(1))} \hcprod \cdots \hcprod 
                    \chi^{U,(\lambda_\ell-\ell + w(\ell))} \\
\label{primary-irreducible-dual-Jacobi-Trudi}
\chi^{U,\lambda}&= \sum_{w \in \Symm_{m}} 
        \sgn(w) \chi^{U,(1^{\lambda'_1-1+w(1)})} \hcprod \cdots \hcprod 
                    \chi^{ U, (1^{\lambda'_m-m + w(m)} ) } 
\end{align}
where  $\chi^{U,(n)}, \chi^{U,(1^n)}$ 
are both the zero character 
if $n < 0$, and the trivial character $\one_{GL_0}$ if $n=0$.

\subsection{The cuspidal characters: indexing and notation}

The set $\CCC_n$ of cuspidal characters for $GL_n(\FF_q)$
has the same cardinality 
$\frac{1}{n} \sum_{d | n} \mu(n/d) q^d$
as the set of irreducible polynomials in $\FF_q[x]$ of degree $n$,
or the set of {\it primitive necklaces} of $n$ beads
having $q$ possible colors (= free orbits under $n$-fold cyclic
rotation of words in $\{0,1,\ldots,q-1\}^n$).
There are at least two ways one sees $\CCC_n$ indexed in the literature.
\begin{itemize}
\item
Green \cite[\S 7]{Green} indexes $\CCC_n$ via free orbits 
$
[\beta]=\{\beta,\beta^q,\cdots,\beta^{q^{n-1}}\}
$ 
for the action of the {\it Frobenius map} 
$\beta \overset{F}{\longmapsto} \beta^q$
on the multiplicative group $\FF_{q^n}^\times$;
he calls such free orbits {\it $n$-simplices}.

In his notation, if $U$ lies in
$\CCC_s$ and is indexed by the orbit $[\beta]$ within $\FF_{q^s}$, 
then the primary $GL_n$-irreducible character 
$\chi^{U,\lambda}$ for a partition
$\lambda$ of $\frac{n}{s}$ 
is (up to a sign) what he denotes $I^\beta_s[\lambda]$.
The special case $I^\beta_s[(m)]$ he also denotes $I^\beta_s[m]$.
Thus the cuspidal $U$ itself is (up to sign) 
denoted $I^\beta_s[1]$, and he also uses 
the alternate terminology $J_s(\beta):=I^\beta_s[1]$; see \cite[p.\ 433]{Green}.
\item
Later authors index $\CCC_n$ via free orbits 
$
[\varphi]=\{\varphi,\varphi \circ F,\cdots, \varphi \circ F^{n-1}\}
$
for the Frobenius action on the {\it dual group $\Hom(\FF^\times_{q^n},\CC^\times)$}.
Say that $U$ in $\CCC_n$ is {\it associated} to the orbit $[\varphi]$ in this indexing.

When $n=1$, one simply has $\CCC_1=\Hom(\FF^\times_{q},\CC^\times)$.
In other words, the Frobenius orbits $[\varphi]=\{\varphi\}$ are
singletons, and if $U$ is associated to this orbit then $U=\varphi$ 
considering both as homomorphisms 
$$
GL_1(\FF_q) = \FF_q^\times \overset{U=\varphi}{\longrightarrow} \CC^\times.
$$
\end{itemize}

Although we will not need Green's full description of 
the characters $\chi^{U,(m)}$ and $\chi^{U,\lambda}$, 
we will use (in the proof of 
Lemma~\ref{normalized-characters-on-semisimple-reflections} below)
the following
consequence of his discussion surrounding \cite[Lemma 7.2]{Green}.

\begin{proposition}
\label{primaries-in-terms-of-cuspidals}
For $U$ in $\CCC_s$, every $\chi^{U,(m)}$, and hence also every
primary irreducible character $\chi^{U,\lambda}$,
is in the $\QQ$-span of characters of the form
$
\chi_{U_1} \hcprod \cdots \hcprod \chi_{U_t}
$
where $U_i$ is in $\CCC_{n_i}$, with $s$ dividing $n_i$ for each $i$.
\end{proposition}

\section{Some explicit character values}
\label{character-values-section}

We will eventually wish to apply Proposition~\ref{conj-count-prop}
with $g$ being a Singer cycle, and with the central elements
$z_i$ being sums over classes of reflections with fixed determinants.
For this one requires explicit character values
on four kinds of conjugacy classes of elements in $GL_n(\FF_q)$:
\begin{itemize}
\item the identity, giving the character degrees,
\item the Singer cycles, 
\item the semisimple reflections, and 
\item the transvections.
\end{itemize}
We review known formulas for most of these, and derive others that we will need, in the next
four subsections.

It simplifies matters that the character value $\chi^{\llambda}(c^{-1})$ vanishes
for a Singer cycle $c$ unless $\chi^{\llambda}=\chi^{U,\lambda}$ is
a primary irreducible character and the partition $\lambda$ of $\frac{n}{s}$ 
takes a very special form; 
see Proposition~\ref{Singer-cycle-character-values} below.  (This may be compared with, for
example, Chapuy and Stump \cite[p.\ 9 and Lemma 5.5]{ChapuyStump}.)

\begin{definition}
The {\it hook-shaped partitions} of $n$ are 
$
\lambda=\hook{k}{n}
$
for $k=0,\ldots,n-1$.
\end{definition}

\noindent
Thus we only compute {\it primary}
irreducible character values, sometimes only those of
the form $\chi^{U,\hook{k}{\frac{n}{s}}}$.

\subsection{Character values at the identity:  the character degrees}

Green computed the degrees of the primary irreducible characters $\chi^{U,\lambda}$ as a
product formula involving familiar quantities associated to partitions.

\begin{definition}
For a partition $\lambda$, recall \cite[(1.5)]{Macdonald} the quantity 
$
n(\lambda):=\sum_{i \geq 1} (i-1)\lambda_i.
$
For a cell $a$ in row $i$ and column $j$ of the Ferrers diagram of $\lambda$
recall the {\it hooklength} \cite[Example I.1]{Macdonald} 
$$
h(a):=h_\lambda(a):=\lambda_i+\lambda'_j-(i+j)+1.
$$
\end{definition}

\begin{theorem}[{\cite[Theorem 12]{Green}}]
\label{dim-formula}
The primary irreducible $GL_n$-character $\chi^{U,\lambda}$
for a cuspidal character $U$ of $GL_s(\FF_q)$ and a partition
$\lambda$ of $\frac{n}{s}$ has degree
\[
\deg( \chi^{U,\lambda} )
=(-1)^{n-\frac{n}{s}} (q;q)_n 
   \frac{q^{s \cdot n(\lambda)}  }{ \prod_{a \in \lambda} (1-q^{s \cdot h(a)})}
=(-1)^{n-\frac{n}{s}} (q;q)_n s_\lambda(1,q^s,q^{2s},\ldots).
\]
\end{theorem}
\noindent
Here $s_\lambda(1,q,q^2,\ldots)$ is the {\it principal specialization}
$x_i=q^{i-1}$ of the Schur function $s_\lambda=s_\lambda(x_1,x_2,\ldots)$.  Observe that this formula 
depends only on $\lambda$ and $s$, and not on the choice of $U\in \CCC_s$.

Two special cases of this formula will be useful in the sequel.
\begin{itemize}
\item The case of hook-shapes:
\begin{equation}
\label{hook-degree-formula}
\deg( \chi^{U, \hook{k}{\frac{n}{s}}} )=
(-1)^{n-\frac{n}{s}} q^{s\binom{k+1}{2}} \frac{(q;q)_n}{(q^s;q^s)_{\frac{n}{s}}}
\qbin{\frac{n}{s}-1}{k}{q^s}.
\end{equation}
\item When $s=1$ and $U=\one$ 
is the trivial character of $GL_1(\FF_q)$, the degree is given by
the usual {\it $q$-hook formula} \cite[\S 7.21]{Stanley}
\begin{equation}
\label{q-hook-formula}
\deg( \chi^{\one,\lambda} )
=f^\lambda(q):=(q;q)_n \frac{q^{n(\lambda)}}{ \prod_{a \in \lambda} (1-q^{h(a)})}
=(q;q)_n s_\lambda(1,q,q^2,\ldots)
=\sum_Q q^{\maj(Q)}
\end{equation}
where the last sum is over all standard Young tableaux $Q$ of shape $\lambda$,
and $\maj(Q)$ is the sum of the entries $i$ in $Q$ for which $i+1$ lies
in a lower row of $Q$.  (Such characters are called \emph{unipotent characters}.)
\end{itemize}

\subsection{Character values on Singer cycles and regular elliptic elements}

Recall from the Introduction that a {\it Singer cycle} in $GL_n(\FF_q)$
is the image of a generator for the cyclic group
$\FF_{q^n}^\times \cong \ZZ/(q^n-1)\ZZ$ under any embedding
$
\FF_{q^n}^\times \hookrightarrow GL_{\FF_q}(\FF_{q^n}) \cong GL_n(\FF_q)
$
that comes from a choice of $\FF_q$-vector space isomorphism 
$\FF_{q^n} \cong \FF_q^n$.  (Such an embedded subgroup $\FF_{q^n}^\times$ is sometimes
called a {\it Coxeter torus} or an {\it anisotropic maximal torus}.)
Many irreducible $GL_n$-character values $\chi^{\llambda}(c^{-1})$ vanish not
only on Singer cycles, but even for a larger class of elements that we introduce in the following
proposition.

\begin{proposition}
\label{regular-elliptic-definition-proposition}
The following are equivalent for $g$ in $GL_n(\FF_q)$.
\begin{enumerate}
\item[(i)] No conjugates $hgh^{-1}$ of $g$ lie in a proper parabolic subgroup $P_\alpha
\subsetneq GL_n$.
\item[(ii)] There are no nonzero proper $g$-stable $\FF_q$-subspaces 
inside $\FF_q^n$.
\item[(iii)] The characteristic 
polynomial $\det(xI_n-g)$ is irreducible in $\FF_q[x]$.
\item[(iv)] The element $g$ is the image of some $\beta$ in $\FF_{q^n}^\times$ satisfying
$\FF_q(\beta)=\FF_{q^n}$ (that is, a {\it primitive element} for $\FF_{q^n}$)
under one of the embeddings
$
\FF_{q^n}^\times \hookrightarrow GL_{\FF_q}(\FF_{q^n}) \cong GL_n(\FF_q).
$
\end{enumerate}
The elements in $GL_n(\FF_q)$ satisfying these properties are called the {\bf regular elliptic
elements}.
\end{proposition}
\begin{proof}
\noindent
{\sf (i) is equivalent to (ii).}
A proper $\FF_q$-subspace $U$, say with $\dim_{\FF_q} U = d < n$, is $g$-stable if and only any 
$h$ in $GL_n(\FF_q)$ sending $U$ to the span of the first $d$ standard
basis vectors in $\FF_q^n$ has the property that $h g h^{-1}$ lies in a proper parabolic subgroup
$P_{\alpha}$ with $\alpha_1=d$.

%
\vskip.1in
\noindent
{\sf (ii) implies (iii).}  
Argue the contrapositive:  if $\det(xI_n-g)$ had a nonzero proper irreducible factor $f(x)$, 
then $\ker(f(g): V \rightarrow V)$ would be a 
nonzero proper $g$-stable subspace.
\vskip.1in
\noindent
{\sf (iii) implies (iv).}  If $f(x):=\det(xI_n-g)$ is irreducible in $\FF_q[x]$,
then $f(x)$ is also the minimal polynomial of $g$.  Thus
$g$ has rational canonical form over $\FF_q$ equal to 
the companion matrix for $f(x)$.  This is the same
as the rational canonical form for the image under
one of the above embeddings of any $\beta$ in $\FF_{q^n}^\times$ having 
minimal polynomial $f(x)$, so that $\FF_q(\beta) \cong \FF_{q^n}$.  
Hence $g$ is conjugate to the image of such an element $\beta$ embedded in 
$GL_n(\FF_q)$, and then equal to such an element, after conjugating the embedding.
\vskip.1in
\noindent
{\sf (iv) implies (ii).}  Assume that $g$ is the image of such an element 
$\beta$ in $\FF_{q^n}^\times$ satisfying $\FF_q(\beta)=\FF_{q^n}$.  Then
a $g$-stable $\FF_q$-subspace $W$ of $\FF_q^n$ would correspond
to a subset of $W \subset \FF_{q^n}$ stable under multiplication by
$\FF_q$ and by $\beta$, so also stable under multiplication 
by $\FF_q(\beta)=\FF_{q^n}$.  This could only be $W=\{0\}$ or $W=\FF_{q^n}$.
\end{proof}

Part (iv) of Proposition~\ref{regular-elliptic-definition-proposition}
shows that Singer cycles $c$ in $G$ 
are always regular elliptic, since
they correspond to elements $\gamma$ for which 
$\FF_{q^n}^\times = \langle \gamma \rangle$, that is, to {\it primitive roots} in
$\FF_{q^n}$.  

\begin{definition}
Recall that associated to the extension
$\FF_{q} \subset \FF_{q^n}$ is the {\it norm map}
$$
\begin{array}{rcl}
\FF_{q^n} &\overset{N_{\FF_{q^n}/\FF_{q}}}{\longrightarrow} & \FF_{q}\\
\beta & \longmapsto &
\beta \cdot \beta^q \cdot \beta^{q^2} \cdots \beta^{q^{n-1}}.
\end{array}
$$
\end{definition}
\noindent
The well-known surjectivity of norm maps for finite fields 
\cite[VII Exer.\ 28]{Lang} is equivalent to the following.

\begin{proposition}
\label{norm-map-preserves-Singer}
If $\FF_{q^n}^\times=\langle \gamma \rangle$,
then $\FF_{q}^\times=\langle N(\gamma) \rangle$. 
\end{proposition}

\begin{proposition}
\label{Singer-cycle-character-values}
Let $g$ be a regular elliptic element in $GL_n(\FF_q)$ associated to $\beta \in \FF_{q^n}$, as in Proposition~\ref{regular-elliptic-definition-proposition}(iv).
\begin{enumerate}
\item[(i)]
The irreducible character $\chi^{\llambda}(g)$ vanishes unless
$\chi^{\llambda}$ is a primary irreducible character $\chi^{U,\lambda}$,
for some $s$ dividing $n$ and some cuspidal character $U$ in $\CCC_s$ and 
partition $\lambda$ in $\Par_{\frac{n}{s}}$.
\item[(ii)]
Furthermore, $\chi^{U,\lambda}(g)=0$ 
except for hook-shaped partitions 
$\lambda=\hook{k}{\frac{n}{s}}$.
\item[(iii)]
More explicitly, if $U$ in $\CCC_s$ is associated to $[\varphi]$
with $\varphi$ in $\Hom(\FF_{q^s}^\times,\CC^\times)$, then
$$
\begin{aligned}
\chi^{U,\hook{k}{\frac{n}{s}}}(g)
&=(-1)^{k} \chi^{U,(\frac{n}{s})}(g)\\ 
&=(-1)^{\frac{n}{s}-k-1} \chi^{U,(1^{\frac{n}{s}})}(g) \\
&=(-1)^{n-\frac{n}{s}-k} \sum_{j=0}^{s-1} 
  \varphi\left( N_{\FF_{q^n}/\FF_{q^s}}(\beta^{q^j}) \right).
\end{aligned}
$$
\item[(iv)]  If in addition $g$ is a Singer cycle then
$$
\sum_{U} \chi^{U,\hook{k}{\frac{n}{s}}}(g) = 
\begin{cases}
(-1)^{n-\frac{n}{s}-k} \mu(s) & \text{ if } q = 2,\\
0 & \text{ if } q \neq 2.
\end{cases}
$$
where the sum is over all $U$ in $\CCC_s$,
and $\mu(s)$ is the usual number-theoretic M\"obius function of $s$.
\end{enumerate}
\end{proposition}

\begin{proof}
The key point is Proposition~\ref{regular-elliptic-definition-proposition}(i),
showing that regular elliptic elements $g$ are the elements whose 
conjugates $hgh^{-1}$ lie in no proper parabolic subgroup $P_\alpha$. 
Hence the parabolic induction formula \eqref{parabolic-induction-formula} shows that
any properly induced character $f_1 \hcprod \cdots \hcprod f_m$ will vanish on a regular
elliptic element $g$.

Assertion (i) follows immediately, as 
\eqref{general-irreducibles-are-induced} shows non-primary irreducibles are
properly induced.

Assertion (ii) also follows, as a non-hook partition $\lambda=(\lambda_1 \geq \lambda_2 \geq
\cdots)$ has
$\lambda_2 \geq 2$, so that in the Jacobi-Trudi-style formula
\eqref{primary-irreducible-Jacobi-Trudi} for 
$\chi^{U,\lambda}$, each term 
$$
\chi^{U,(\lambda_1-1+w(1))} \hcprod 
\chi^{U,(\lambda_2-2+w(2))} \hcprod \cdots \hcprod
\chi^{U,(\lambda_\ell-\ell+w(\ell))}
$$
begins with two nontrivial product factors, so it
is properly induced, and vanishes on regular elliptic $g$.

The first two equalities asserted in (iii) follow from similar analysis of  
terms in \eqref{primary-irreducible-Jacobi-Trudi},
\eqref{primary-irreducible-dual-Jacobi-Trudi} 
for $\chi^{U,\lambda}$ when $\lambda=\hook{k}{\frac{n}{s}}$.  These 
formulas have $2^{k+1}$ and $2^{\frac{n}{s}-k}$ nonvanishing terms, 
respectively, of the form
$$
\begin{array}{rcl}
&(-1)^{k-m} &
\chi^{U,(\alpha_1)}  \hcprod 
\chi^{U,(\alpha_2)}  \hcprod 
\cdots \hcprod
\chi^{U,(\alpha_m)}   \\
&(-1)^{\frac{n}{s}-k-m}& 
\chi^{U,(1^{\beta_1})}  \hcprod 
\chi^{U,(1^{\beta_2})}  \hcprod 
\cdots \hcprod
\chi^{U,(1^{\beta_m})}   
\end{array}
$$
corresponding to compositions $(\alpha_1,\ldots,\alpha_m)$ 
and $(\beta_1,\ldots,\beta_m)$ of $\frac{n}{s}$
with $\alpha_1 \geq k+1$ and $\beta_1 \geq \frac{n}{s}-k$, respectively.
All such terms vanish on regular elliptic $g$, 
being properly induced, except the $m=1$ terms:
$$
\begin{aligned}
\chi^{U,\hook{k}{\frac{n}{s}}}(g)
&=(-1)^{k-1} \chi^{U,(\frac{n}{s})}(g)\\ 
&=(-1)^{\frac{n}{s}-k-1} \chi^{U,(1^{\frac{n}{s}})}(g).
\end{aligned}
$$
The last equality in (iii) comes from a result of 
Silberger and Zink \cite[Theorem 6.1]{SilbergerZink},
which they deduced by combining various formulas from Green \cite{Green}.

For assertion (iv), say that the regular elliptic element
$g$ corresponds to an element $\beta$ in $\FF_{q^n}$ under the 
embedding $\FF_{q^n}^\times \hookrightarrow GL_n(\FF_q)$, and let
$\gamma:=N_{\FF_{q^n}/\FF_{q^s}}(\beta)$ be its norm in 
$\FF_{q^s}^\times$. 
Assertion (iii) and
the multiplicative property of the norm map $N_{\FF_{q^n}/\FF_{q^s}}$ imply
\begin{equation}
\label{norm-orbit-Singer-sum}
\sum_{U} \chi^{U,\hook{k}{\frac{n}{s}}}(g) 
=(-1)^{n-\frac{n}{s}-k} 
\sum_U \sum_{j=0}^{s-1} 
  \varphi\left( \gamma^{q^j} \right) \\
=(-1)^{n-\frac{n}{s}-k} 
\sum_U \sum_{\gamma'} \varphi(\gamma')
\end{equation}
where the inner sum runs over all $\gamma'$ lying in the Frobenius
orbit of $\gamma$ within $\FF_{q^s}$.  When
 one further assumes that $g$ is a Singer cycle,
then Proposition~\ref{norm-map-preserves-Singer} implies
$\FF_{q^s}^\times = \langle \gamma \rangle$, so that a homomorphism
$\varphi: \FF_{q^s}^\times \rightarrow \CC^\times$ is completely determined
by its value $z:=\varphi(\gamma)$ in $\CC^\times$.
Furthermore, $\varphi$ will have a free Frobenius orbit if
and only if the powers $\{z, z^q, z^{q^2},\cdots,z^{q^{s-1}}\}$ 
are {\it distinct} roots of unity.  
Thus one can rewrite the rightmost summation
$\sum_U \sum_{\gamma'} \varphi(\gamma')$ in \eqref{norm-orbit-Singer-sum} above
as the sum over all $z$ in $\CC^\times$ for which $z^{q^s-1}=1$
but $z^{q^t-1} \neq 1$ for any proper divisor $t$ of $s$.
Number-theoretic M\"obius inversion shows this is
$
\sum_{t | s} \mu\left( \frac{s}{t} \right)
f(t)
$
where 
\[
f(t):=\sum_{\substack{z \in \CC^\times:\\ z^{q^t-1}=1}} z
=\begin{cases}
1 & \text{ if } q = 2, t=1,\\
0 & \text{ if } q \neq 2 \text{ or }t \neq 1.
\end{cases}
\]
Hence 
\begin{equation}
\label{moebius-function-equality}
\sum_U \sum_{\gamma'} \varphi(\gamma')=
\begin{cases}
\mu(s) & \text{ if } q = 2,\\
0 & \text{ if } q \neq 2.          
\end{cases}
\qedhere
\end{equation}
\end{proof}

\subsection{Character values on semisimple reflections}
\label{semisimple-reflection-character-section}

Recall that a semisimple reflection $t$ in $GL_n(\FF_q)$ has
conjugacy class determined by its non-unit eigenvalue $\det(t)$, lying in $\FF_q^\times \setminus
\{1\}$.  Recall also the notion of the {\it content} $c(a):=j-i$ of a cell $a$ lying in row $i$ and
column $j$ of
the Ferrers diagram for a partition $\lambda$.

\begin{lemma}
\label{normalized-characters-on-semisimple-reflections}
Let $t$ be a semisimple reflection in 
$GL_n(\FF_q)$.

\begin{enumerate}
\item[(i)] Primary irreducible characters $\chi^{U,\lambda}$ vanish
on $t$ unless $\wt(U)=1$, that is, unless $U$ is in $\CCC_1$.
\item[(ii)] For $U$ in $\CCC_1$, so  
$
\FF_q^\times \overset{U}{\rightarrow} \CC^\times,
$
and $\lambda$ in $\Par_n$, the normalized character $\normchi^{U,\lambda}$ has value on  $t$
$$
\normchi^{U,\lambda}(t)
= U(\det(t)) \cdot
    \frac{\sum_{a \in \lambda} q^{c(a)}}{[n]_q}.
$$
\item[(iii)] In particular, for $U$ in $\CCC_1$ and hook shapes $\lambda=(n-k,1^k)$,
this simplifies to 
$$
\normchi^{U,(n-k,1^k)}(t)=U(\det(t)) \cdot q^{-k}.
$$ 

\end{enumerate}
\end{lemma}
\begin{proof}
For assertion (i), we start with the fact proven by 
Green \cite[\S 5 Example (ii), p. 430]{Green} that cuspidal characters for $GL_n$ 
vanish on {\it non-primary} conjugacy classes, that is,
those for which the characteristic polynomial is divisible by at least
two distinct irreducible polynomials in $\FF_q[x]$.  

This implies cuspidal characters
for $GL_n$ with $n \geq 2$ vanish on semisimple reflections $t$, since
such $t$ are non-primary:  $\det(xI-t)$ is divisible
by both $x-1$ and $x-\alpha$ where $\alpha=\det(t) \neq 1$.

Next, the parabolic induction formula \eqref{parabolic-induction-formula}
shows that any character of the form $\chi_{U_1} \hcprod \cdots \hcprod \chi_{U_\ell}$
in which each $U_i$ is a $GL_{n_i}$-cuspidal with $n_i \geq 2$ will also vanish
on all semisimple reflections $t$:  whenever $hth^{-1}$ lies in the parabolic
$P_{(n_1,\ldots,n_\ell)}$ and has diagonal blocks $(g_1,\ldots,g_\ell)$, one
of the $g_{i_0}$ is also a semisimple reflection, so that 
$\chi_{U_{i_0}}(g_{i_0})=0$ by the above discussion.

Lastly, Lemma~\ref{primaries-in-terms-of-cuspidals} shows that every 
primary irreducible $\chi^{U,\lambda}$ with $\wt(U) \geq 2$ will vanish
on every semisimple reflection:  $\chi^{U,\lambda}$ is in the $\QQ$-span
of characters $\chi_{U_1} \hcprod \cdots \hcprod \chi_{U_\ell}$ with each $U_i$ a
$GL_{n_i}$-cuspidal in which $\wt(U)$ divides $n_i$, so that $n_i \geq 2$.

Assertion (iii) is an easy calculation using assertion (ii), so it only remains to prove (ii).
We first claim that one can reduce to the case where character $U$ in $\CCC_1$ is the
trivial character $\FF_q^\times \overset{U=\one}{\longrightarrow} \CC^\times$.  This is
because one has
$
\chi^{U,(n)} = U = U \otimes \chi^{\one,(n)}
$
and hence using \eqref{primary-irreducible-Jacobi-Trudi} one has
\begin{equation}
\label{weight-one-cuspidal-tensor-product}
\chi^{U,\lambda} =  U \otimes \chi^{\one,\lambda} 
\qquad\text{ for }\lambda\text{ in }\Par_n\text{ when }U\text{ lies in }\CCC_1.
\end{equation}

Thus without loss of generality, $U=\one$, and we wish to show
\begin{equation}
\label{semisimple-refn-character-value-at-unipotent}
\normchi^{\one,\lambda}(t) 
= \frac{1}{[n]_q} \sum_{a \in \lambda} q^{c(a)}.
\end{equation}

\begin{lemma}
\label{semisimple-reflection-character-skews}
A semisimple reflection $t$ has 
$
\chi^{\one,\lambda}(t) = \Psi( s_\lambda )
$
where $\Psi$ is the linear map on the symmetric functions $\Lambda=\QQ[p_1,p_2,\ldots]$ 
expressed in terms of power sums that sends 
$
f(x_1,x_2,\ldots)  \mapsto  (q;q)_{n-1} \frac{\partial f}{\partial p_1}(1,q,q^2,\ldots).
$
\end{lemma}
\begin{proof}[Proof of Lemma~\ref{semisimple-reflection-character-skews}]
By linearity and \eqref{primary-irreducible-Jacobi-Trudi},
it suffices to check for compositions $\alpha=(\alpha_1,\ldots,\alpha_m)$ of $n$ that 
$
\chi^{\one,\alpha}:=\chi^{\one,(\alpha_1)} \hcprod \cdots \hcprod \chi^{\one,(\alpha_m)} 
$
has $\chi^{\one,\alpha}(t) = \Psi (h_\alpha)$
where $h_\alpha = h_{\alpha_1} \cdots h_{\alpha_m}$.  
The character $\chi^{\one,\alpha}$ 
is just the usual induced character 
$\Ind_{P_\alpha}^{GL_n} \one_{P_\alpha}$, so the
permutation character on the set of 
{\it $\alpha$-flags of subspaces}
$$
\{0\} \subset V_{\alpha_1} \subset V_{\alpha_1+\alpha_2} \subset \cdots \subset
V_{\alpha_1+\alpha_2+\cdots+\alpha_{m-1}} \subset \FF_q^n,
$$ 
which are counted by the $q$-multinomial coefficient
$$
\qbin{n}{\alpha}{q}:=\qbin{n}{\alpha_1,\ldots,\alpha_m}{q} 
= \frac{[n]!_q}{[\alpha_1]!_q \cdots [\alpha_m]!_q}
=(q;q)_n h_\alpha(1,q,q^2,\ldots).
$$
Thus $\chi^{\one,\alpha}(t)$ counts the number of such flags stabilized by the semisimple
reflection $t$.  To count these let $H$ and $L$ denote, respectively, the fixed hyperplane for $t$
and the line which is the $\det(t)$-eigenspace for $t$.  Then one can classify the
$\alpha$-flags stabilized by $t$ according to the smallest index $i$
for which $L \subset V_{\alpha_1+\cdots+\alpha_i}$.
Fixing this index $i$, such flags must have their first $i-1$ subspaces 
$V_{\alpha_1}, V_{\alpha_1+\alpha_2},\ldots,V_{\alpha_1+\cdots+\alpha_{i-1}}$
lying inside $H$, and 
their remaining subspaces from $V_{\alpha_1+\cdots+\alpha_i}$ onward 
containing $L$.  From this description it is not hard to see that the quotient map
$\FF_q^n \twoheadrightarrow \FF_q^n/L$
is a bijection between such $t$-stable $\alpha$-flags and the 
$(\alpha-e_i)$-flags in $\FF_q^n/L \cong \FF_q^{n-1}$, where
$
\alpha-e_i:=(\alpha_1,\ldots,\alpha_{i-1},
             \alpha_i-1,\alpha_{i+1},\ldots,\alpha_{m}).
$
Consequently
$$
\begin{aligned}
\chi^{\one,\alpha}(t) = \sum_{i=1}^m \qbin{n-1}{\alpha-e_i}{q} 
=(q;q)_{n-1} \sum_{i=1}^m  h_{\alpha-e_i}(1,q,q^2,\ldots) 
= (q;q)_{n-1} \frac{\partial h_{\alpha}}{\partial p_1}(1,q,q^2,\ldots)
=\Psi(h_\alpha) 
\end{aligned}
$$
using the fact \cite[Example I.5.3]{Macdonald}
that $\frac{\partial h_n}{\partial p_1}=h_{n-1}$,
and hence $\frac{\partial h_{\alpha}}{\partial p_1} = \sum_{i=1}^m h_{\alpha-e_i}$
via the Leibniz rule.
\end{proof}

Resuming the proof of \eqref{semisimple-refn-character-value-at-unipotent}, since
\cite[Example I.5.3]{Macdonald} shows
$
\partial s_\lambda/\partial p_1
  = \sum_{\mu \subset \lambda:\\ |\mu|=|\lambda|-1} s_\mu,
$
one concludes from Lemma~\ref{semisimple-reflection-character-skews} 
and \eqref{q-hook-formula} that
$$
\normchi^{\one, \lambda}(t) 
= \frac{\chi^{\one, \lambda}(t)}{\deg(\chi^{\one,\lambda})}
= \sum_{\substack{\mu \subset \lambda:\\ |\mu|=|\lambda|-1}} 
\frac{(q;q)_{n-1} s_\mu(1,q,q^2,\ldots)}{(q;q)_n s_\lambda(1,q,q^2,\ldots)} \\
= \sum_{\substack{\mu \subset \lambda:\\ |\mu|=|\lambda|-1}}
\frac{ f^{\mu}(q) }{ f^\lambda(q) }
$$
where $f^\lambda(q)$ is the $q$-hook formula from \eqref{q-hook-formula}.
Thus the desired equation \eqref{semisimple-refn-character-value-at-unipotent}
becomes the assertion
\begin{equation}
\label{q-hook-walk-formula}
\sum_{\substack{\mu \subset \lambda: \\ |\mu|=|\lambda|-1}} 
    \frac{ f^{\mu}(q) }{ f^\lambda(q) } 
= \frac{1}{[n]_q}\sum_{a \in \lambda} q^{c(a)}
\end{equation}
which follows from either of two results in the literature:
\eqref{q-hook-walk-formula}
is equivalent\footnote{In checking this equivalence, it is useful to bear in mind that 
$f^{\lambda^t}(q) = q^{\binom{n}{2}} f^{\lambda}(q^{-1})$, along with the fact that
if $\mu \subset \lambda$ with $|\mu|=|\lambda|-1$ and the unique cell of $\lambda/\mu$ lies in
row $i$ and column $j$,
then $n(\lambda)-n(\mu)=i-1$ and $n(\lambda^t)-n(\mu^t)=j-1$.}, 
after sending $q \mapsto q^{-1}$, to a result of Kerov \cite[Theorem 1 and Eqn.\ (2.2)]{Kerov},
and \eqref{q-hook-walk-formula}
 is also the $t=q^{-1}$ specialization of a result of
Garsia and Haiman \cite[(I.15), Theorem 2.3]{GarsiaHaiman}.
\end{proof}

\subsection{Character values on transvections}
\label{transvection-character-section}

The $GL_n$-irreducible character values on transvections appear in probabilistic
work of M.~Hildebrand \cite{Hildebrand}.  For primary irreducible characters, his result is
equivalent\footnote{In seeing this equivalence, note that Hildebrand uses Macdonald's indexing
\cite[p. 278]{Macdonald} of $GL_n$-irreducibles, where partition values are transposed in the
functions 
$\llambda: \CCC \longrightarrow \Par$ relative to our convention in
Sections~\ref{irreducible-parametrization-section} 
and \ref{Jacobi-Trudi-section}.}
to the following.

\begin{theorem}[{\cite[Theorem 2.1]{Hildebrand}}]
\label{Hildebrand's-calculation}
For $U$ in $\CCC_s$ with $\lambda$ in $\Par_{\frac{n}{s}}$, a transvection $t$ in
$GL_n(\FF_q)$
has 
$$
\normchi^{U,\lambda}(t)= 
\begin{cases}
\displaystyle
\frac{1}{1 - q^{n-1}}
\left( 1 - q^{n-1}
  \sum_{\substack{\mu \subset \lambda: \\ |\mu|=|\lambda|-1}} 
     \frac{ f^{\mu}(q) }{ f^\lambda(q) } 
 \right)
  & \text{ if }s=1,\\
\displaystyle
\frac{1}{1 - q^{n-1}} 
  & \text{ if } s \geq 2.
\end{cases}
$$
\end{theorem}

\noindent
One can rephrase the $s=1$ case similarly to 
Lemma~\ref{normalized-characters-on-semisimple-reflections}(ii).

\begin{corollary}
\label{normalized-characters-on-transvections}
For $U$ in $\CCC_1$ with $\lambda$ in $\Par_{n}$, 
a transvection $t$ in $GL_n(\FF_q)$ has 
$$
\normchi^{U,\lambda}(t)
= \frac{
\displaystyle
1 - q^{n-1} \left( \frac{ \sum_{a \in \lambda} q^{c(a)} }{[n]_q} \right)
}
{1 - q^{n-1}}.
$$
In particular, for $U$ in $\CCC_1$ and $0 \leq k \leq n-1$, one has
$$
\normchi^{U,\hook{k}{n}}(t)=
\frac{1 - q^{n-k-1}}{1 - q^{n-1}}.
$$
\end{corollary}
\begin{proof}
The first assertion follows from Theorem~\ref{Hildebrand's-calculation} using
\eqref{q-hook-walk-formula}, and the second from the calculation
\[
\sum_{a \in \hook{k}{n}} q^{c(a)} = q^{-k} + q^{-k+1} + \cdots + q^{n-k-1}
= q^{-k} [n]_q. \qedhere
\]
\end{proof}

\section{Proofs of Theorems~\ref{q-factorization-theorem} and
\ref{fixed-det-sequence-theorem}.}
\label{main-result-proof-section}

In proving the main results Theorems~\ref{q-factorization-theorem} and
\ref{fixed-det-sequence-theorem}, it is convenient to know the equivalences between the various
formulas that they assert.
After checking this in Proposition~\ref{q-Jackson-sum-equals-tot-num-lemma},
we assemble the normalized character values on reflection conjugacy class sums, 
in the form needed to apply
\eqref{Chapuy-Stump-varying-class-answer}.  This is then used to prove
Theorem~\ref{fixed-det-sequence-theorem} for $q>2$, from which we derive
Theorem~\ref{q-factorization-theorem} for $q>2$.  Lastly we prove
Theorem~\ref{q-factorization-theorem} for $q=2$.

\subsection{Equivalences of the formulas}
We will frequently use the easy calculation 
\begin{equation}
\label{q-difference-on-powers}
\left[\Delta_q^{N} \left( x^A \right) \right]_{x=1}
= \frac{(q^{A-N+1};q)_N}{(1-q)^N} 
\end{equation}
which can be obtained by iterating $\Delta_q$, or
via \eqref{q-diff-iterate} and the {\it $q$-binomial theorem}
\cite[p.\ 25, Exer.\ 1.2(vi)]{GasperRahman} 
\begin{equation}
\label{q-binomial-theorem}
(x;q)_N=\sum_{k=0}^N (-x)^k q^{\binom{k}{2}}\qbin{N}{k}{q}.
\end{equation}

The following assertion was promised in the Introduction.

\begin{proposition}
\label{q-Jackson-sum-equals-tot-num-lemma} 
As polynomials in $q$,
\begin{itemize}
\item[(i)] the three expressions 
\eqref{q-Jackson-sum-formula}, \eqref{q-Jackson-difference-formula}, \eqref{tot-num-formula}
for $t_q(n,\ell)$ asserted in Theorem~\ref{q-factorization-theorem}
all agree, and
\item[(ii)] the two expressions  \eqref{det-sequence-difference-formula}, \eqref{q-diff-nlm} 
 for $t_q(n,\ell,m)$ asserted in Theorem~\ref{fixed-det-sequence-theorem} 
agree if $m\le \ell-1.$
\end{itemize}
\end{proposition}

\begin{proof}
\noindent
{\sf Assertion (i).}
Starting with \eqref{q-Jackson-difference-formula}
$$
t_q(n,\ell)
 = (1-q)^{-1}\frac{(-[n]_q)^\ell}{[n]!_q} \left[ \Delta_q^{n-1} 
        \biggl(\frac{1}{x}-\frac{(1+x(1-q))^\ell}{x}\biggr)\right]_{x=1},
$$
linearity of the operator 
$g(x) \longmapsto \left[\Delta_q^{n-1} g(x) \right]_{x=1}$
lets one expand in two different ways its subexpression
\begin{equation}
\label{q-Jackson-difference-formula-subexpression}
\left[ \Delta_q^{n-1} 
        \biggl(\frac{1}{x}-f(x) \biggr)\right]_{x=1}
\qquad \text{ where }f(x):=\frac{\left(1+x(1-q)\right)^\ell}{x}.
\end{equation}

The first way will yield \eqref{q-Jackson-sum-formula},
by expanding \eqref{q-Jackson-difference-formula-subexpression} as 
$\left[\Delta_q^{n-1} \left( \frac{1}{x} \right) \right]_{x=1}-
\left[\Delta_q^{n-1} f(x)  \right]_{x=1}.
$
Note that
$$
\left[\Delta_q^{n-1} \left( \frac{1}{x} \right) \right]_{x=1}
= \frac{(q^{1-n};q)_{n-1}}{(1-q)^{n-1}} 
= \frac{(-1)^{n-1}}{q^{\binom{n}{2}} (1-q)^{n-1}} (q;q)_{n-1}
$$
via \eqref{q-difference-on-powers}, which accounts for
the $(-1)^{n-1}(q;q)_{n-1}$ term inside the large parentheses of
\eqref{q-Jackson-sum-formula}.
Meanwhile, applying \eqref{q-diff-iterate}
to $\left[\Delta_q^{n-1} f(x) \right]_{x=1}$ and noting that
$
f(q^{n-1-k})=q^{1-n}q^{k}(1+q^{n-k-1}-q^{n-k})^\ell,
$
one obtains a summation that accounts for the remaining terms 
inside the large parentheses of \eqref{q-Jackson-sum-formula}.
This shows the equivalence of \eqref{q-Jackson-sum-formula},
\eqref{q-Jackson-difference-formula}.
The second way will yield  \eqref{tot-num-formula},
by expanding
$f(x)=\sum_{i=0}^{\ell} \binom{\ell}{i}(1-q)^{\ell-i} x^{\ell-i-1}$,
and noting that the $i=\ell$ term cancels with the $\frac{1}{x}$ appearing
inside \eqref{q-Jackson-difference-formula-subexpression}.
Therefore \eqref{q-Jackson-difference-formula} becomes
$$
\begin{aligned}
t_q(n,\ell)=& (-[n]_q)^\ell \frac{(1-q)^{n-1}}{(q;q)_n}
\sum_{i=0}^{\ell-1} -\binom{\ell}{i}(1-q)^{\ell-i} \left[ \Delta_q^{n-1}\left( x^{\ell-i-1} \right)
\right]_{x=1}\\
=& 
[n]_q^{\ell-1} \sum_{i=0}^{\ell-n} (-1)^i(q-1)^{\ell-i-1}\binom{\ell}{i}\qbin{\ell-i-1}{n-1}{q}
,
\end{aligned}
$$
using \eqref{q-difference-on-powers}.
The summands with $\ell-n+1\le i\le \ell-1$ vanish, 
showing the equivalence of \eqref{q-Jackson-difference-formula}, \eqref{tot-num-formula}. 

\vskip.1in
\noindent
{\sf Assertion (ii).}
Starting with \eqref{q-diff-nlm}, 
$$
t_q(n,\ell,m)
= \frac{[n]_q^\ell}{[n]!_q} 
\left[ \Delta_q^{n-1}\bigl( (x-1)^m x^{\ell-m-1}\bigr) \right]_{x=1},
$$
expand the $(x-1)^m$ factor via the binomial theorem.  
Using \eqref{q-difference-on-powers}, this expression for $t_q(n,\ell,m)$ becomes
\[
t_q(n,\ell,m)
= [n]_q^{\ell-1} \sum_{i = 0}^m    (-1)^i \,  \binom{m}{i}
               \,  \qbin{\ell-i - 1}{n - 1}{q}    .         
\]
As $i\le m\le \ell-1$, one has $\ell-i-1\ge 0$ and the sum is actually over
$0 \le i \le \ell-n$, agreeing with \eqref{det-sequence-difference-formula}.
\end{proof}

\subsection{The normalized characters on reflection conjugacy class sums}

\begin{definition}
\label{reflection-conjugacy-class-sum-definition}
For $\alpha$ in $\FF_q^\times$,
let $z_\alpha:=\sum_{t:\det(t)=\alpha} t$ in $\CC GL_n$ be the sum of 
reflections of determinant $\alpha$.  
\end{definition}

\begin{corollary}
\label{normalized-character-on-reflection-class} 
For $U$ in $\CCC_s$, and $k$ in the range $0 \leq k \leq \frac{n}{s}$, 
and any $\alpha$ in $\FF_q^\times \setminus \{ 1 \}$,
one has
\begin{eqnarray}
\label{semisimple-reflection-class-value}
\normchi^{U,\hook{k}{\frac{n}{s}}}(z_\alpha)
&=&
[n]_q  \left. \begin{cases}
q^{n-k-1} U(\alpha) &\text{ if }s=1 \\
0            &\text{ if }s \geq 2.
\end{cases} \right\},\\
\label{transvection-class-value}
\normchi^{U,\hook{k}{\frac{n}{s}}}(z_1) 
&=&
[n]_q \left. \begin{cases}
q^{n-k-1}-1 &\text{ if }s=1, \\
-1            &\text{ if }s \geq 2.
\end{cases} \right\}.
\end{eqnarray}

\end{corollary}
\begin{proof}
First we count the reflections $t$ in $GL_n(\FF_q)$.
There are $[n]_q=1+q+q^2+\cdots+q^{n-1}$ choices for the hyperplane $H$ fixed
by $t$.  To count the reflections fixing $H$, without loss of generality one
can conjugate $t$ and assume that
$H$ is the hyperplane spanned by the first $n$ standard basis vectors
$e_1,\ldots,e_{n-1}$.

If $t$ is a semisimple reflection then its
conjugacy class is determined by its determinant, lying
in $\FF_q^\times \setminus \{1\}$.  Having fixed $\alpha:=\det(t)$,
there will be $q^{n-1}$ such reflections that fix $e_1,\ldots,e_{n-1}$:
each is determined by sending $e_n$ to $\alpha e_n + \sum_{i=1}^{n-1} c_i e_i$
for some  $(c_1,\ldots,c_{n - 1})$ in $\FF_q^{n-1}$.
Hence \eqref{semisimple-reflection-class-value} follows from
Lemma~\ref{normalized-characters-on-semisimple-reflections}.

The nonsemisimple reflections $t$ are the transvections,
forming a single conjugacy class, with $\det(t)=1$.
There will be $q^{n-1}-1$ transvections that fix $e_1,\ldots,e_{n-1}$:
each is determined by sending $e_n$ to $e_n + \sum_{i=1}^{n-1} c_i e_i$
for some $(c_1, \ldots, c_{n-1})$ in 
$\FF_q^{n-1} \setminus \{ \mathbf{0} \}$.
Hence \eqref{transvection-class-value} follows from
Theorem~\ref{Hildebrand's-calculation} and 
Corollary~\ref{normalized-characters-on-transvections}.
\end{proof}



\subsection{Proof of Theorem~\ref{fixed-det-sequence-theorem} for $q>2$.}

For a Singer cycle $c$ in $GL_n(\FF_q)$, and 
$\alpha=(\alpha_1,\ldots,\alpha_\ell)$ in $(\FF_q^\times)^{\ell}$ 
with $\prod_{i=1}^\ell\alpha_i=\det(c)$, Proposition~\ref{conj-count-prop} counts 
the reflection factorizations
$c=t_1 t_2 \cdots t_\ell$ with $\det(t_i)=\alpha_i$ as
\begin{equation}\label{factorizations with fixed alpha}
\frac{1}{|GL_n|} \sum_{\substack{(s,U):\\s | n\\U \in \CCC_s}}
      \sum_{k=0}^{\frac{n}{s}-1} 
      \deg(\chi^{U,\hook{k}{\frac{n}{s}}}) \cdot 
      \chi^{U,\hook{k}{\frac{n}{s}}} (c^{-1}) 
      \cdot \prod_{i=1}^\ell \normchi^{U,\hook{k}{\frac{n}{s}}}(z_{\alpha_i}).
\end{equation}
There are several simplifications in this formula.

Firstly, note that the outermost sum
over pairs $(s,U)$ reduces to the pairs with $s=1$:  since
$\det(c)$ is a primitive root in $\FF_q^\times$ by 
Proposition~\ref{norm-map-preserves-Singer} 
and $q>2$, one knows that $\det(c) \neq 1$, so
that at least one of the $\alpha_i$ is not $1$.  Thus its factor
$\normchi^{U,\hook{k}{\frac{n}{s}}}(z_{\alpha_i})$ in the last product
will vanish if $s \geq 2$ by \eqref{semisimple-reflection-class-value}.

Secondly, when $s=1$ then Corollary~\ref{normalized-character-on-reflection-class}
evaluates the product in \eqref{factorizations with fixed alpha} as
\begin{equation}
\label{varying-class-product-with-det}
\prod_{i=1}^\ell \normchi^{U,\hook{k}{\frac{n}{s}}}(z_{\alpha_i})
=
[n]_q^\ell \,\, (q^{n-k-1}-1)^m \,\, q^{(n-k-1)(\ell-m)} \,\, U(\det(c))
\end{equation}
if exactly $m$ of the $\alpha_i$ are equal to $1$, that is,
if the number of transvections in the factorization is $m$.
This justifies calling it $t_q(n,\ell,m)$ where $m\le \ell-1$.

Thirdly, for $s=1$ Proposition~\ref{Singer-cycle-character-values}(iii) 
shows\footnote{Here we use the fact that $c^{-1}$ is also a Singer cycle.}
that
$
\chi^{U,\hook{k}{n}} (c^{-1})  =  (-1)^k U(\det(c^{-1}))
$,
so there will be cancellation of the factor $U(\det(c))$
occurring in \eqref{varying-class-product-with-det}
within each summand of \eqref{factorizations with fixed alpha}.

Thus plugging in the degree formula
from the $s=1$ case of \eqref{hook-degree-formula},
one obtains the following formula for \eqref{factorizations with fixed alpha},
which we denote by $t_q(n,\ell,m)$, emphasizing its
dependence only on $\ell$ and $m$, not on the sequence $\alpha$:

\[
t_q(n,\ell,m)
= \frac{(q - 1)[n]_q^\ell}{|GL_n|} 
      \sum_{k=0}^{n-1} 
      q^{\binom{k + 1}{2}}\qbin{n-1}{k}{q} \,  
      (-1)^k 
      \,  (q^{n-k-1}-1)^m \,  q^{(n-k-1)(\ell-m)}.
\]
This expression may be rewritten using the $q$-difference operator $\Delta_q$ and 
\eqref{q-diff-iterate} as
\[
t_q(n,\ell,m)
= \frac{[n]_q^\ell}{|GL_n|} q^{\binom{n}{2}}(q-1)^n
\left[ \Delta_q^{n-1}\bigl( (x-1)^mx^{\ell-m-1}\bigr) \right]_{x=1}.
\]
Since $|GL_n|=q^{\binom{n}{2}} (-1)^n (q;q)_n$, this
last expression is the same as \eqref{q-diff-nlm}.  
Hence by Proposition~\ref{q-Jackson-sum-equals-tot-num-lemma},
this completes the proof of Theorem~\ref{fixed-det-sequence-theorem} 
for $q > 2$.

that the final sum has all summands $0$ except for the $i = 0$ summand, whence in this case we
have $[n]_q^{n - 1}$ reflections.  In particular, for $\ell = n$ the number of factorizations is
completely independent of the tuple $\alpha$ of determinants (provided the product of the entries
of $\alpha$ actually is $\det c$).

\subsection{Proof of Theorem~\ref{q-factorization-theorem} when $q>2$.}

We will use Theorem~\ref{fixed-det-sequence-theorem} for $q>2$ to derive
\eqref{q-Jackson-difference-formula} for $q>2$.
First note that one can choose a sequence of determinants
 $\alpha=(\alpha_1,\ldots,\alpha_\ell)$ in $\FF_q^\times$ that
has $\prod_{i=1}^\ell \alpha_i =\det(c)$ and has 
exactly $m$ of the $\alpha_i=1$ in a two-step process:
first choose $m$ positions out of $\ell$
to have $\alpha_i=1$, then choose the remaining sequence in 
$\left( \FF_q^\times\setminus\{1\} \right)^{\ell-m}$
with product equal to $\det(c)$. 
Simple counting shows that in a finite group $K$,
the number of sequences in $(K \setminus \{1\})^N$ whose
product is some fixed nonidentity element\footnote{In fact, 
Theorem~\ref{q-factorization-theorem} is stated for $n \geq 2$, but
remains valid for when $n=1$ and $q > 2$.  It is only in the trivial
case where $GL_1(\FF_2)=\{1\}$ that the ``Singer cycle'' $c$ is actually the
{\it identity element}, so that the count \eqref{factoring-non-identity-elements-count} fails.} of 
$K$ is
\begin{equation}
\label{factoring-non-identity-elements-count}
\frac{(|K|-1)^{N}-(-1)^{N}}{|K|}.
\end{equation}
Applying this to $K=\FF_q^\times$ with $N=\ell-m$ gives
$$
t_q(n,\ell)
=\sum_{m=0}^{\ell} t_q(n,\ell,m) \binom{\ell}{m} 
      \frac{(q-2)^{\ell-m}-(-1)^{\ell-m}}{q-1}.
$$
Thus from \eqref{q-diff-nlm} one has
\begin{align*}
t_q(n,\ell)
=& \frac{(q-1)[n]_q^\ell}{|GL_n|}q^{\binom{n}{2}}(q-1)^{n-1}
\left[ 
   \Delta_q^{n-1} \left(\sum_{m=0}^{\ell} \binom{\ell}{m} (x-1)^{m}x^{\ell-m-1}
      \frac{(q-2)^{\ell-m}-(-1)^{\ell-m}}{q-1}\right)
  \right]_{x=1}\\
=& \frac{(-[n]_q)^\ell}{|GL_n|}q^{\binom{n}{2}}(q-1)^{n-1}
    \left[ \Delta_q^{n-1} \biggl(\frac{(1+x(1-q))^\ell}{x}-\frac{1}{x}\biggr) \right]_{x=1}\\
=& (1-q)^{-1}\frac{(-[n]_q)^\ell}{[n]!_q} \left[ \Delta_q^{n-1} 
        \biggl(\frac{1}{x}-\frac{(1+x(1-q))^\ell}{x}\biggr)\right]_{x=1}
\end{align*}
which is \eqref{q-Jackson-difference-formula}.  Hence by 
Proposition~\ref{q-Jackson-sum-equals-tot-num-lemma},
this completes the proof of Theorem~\ref{q-factorization-theorem} when $q>2$.

\subsection{Proof of Theorem~\ref{q-factorization-theorem} when $q=2$.}

Here all reflections are transvections and \eqref{Chapuy-Stump-varying-class-answer}
gives us
$$
\begin{aligned}
t_q(n,\ell)
&=\frac{1}{|GL_n|} 
\sum_{\chi^\llambda \in \Irr(GL_n)} 
  \deg(\chi^{\llambda}) \cdot \chi^{\llambda}(c^{-1}) 
      \cdot \normchi^{\llambda}(z_1)^\ell \\
&=\frac{1}{|GL_n|} \sum_{\substack{(s,U):\\s | n\\U \in \CCC_s}}
     \underbrace{
      \sum_{k=0}^{\frac{n}{s}-1} 
      \deg(\chi^{U,\hook{k}{\frac{n}{s}}}) \cdot 
      \chi^{U,\hook{k}{\frac{n}{s}}} (c^{-1}) 
      \cdot \normchi^{U,\hook{k}{\frac{n}{s}}}(z_1)^\ell
     }_{\text{Call this } f(s,U)}
\end{aligned}
$$
using the vanishing of $\chi^{\llambda}(c^{-1})$
from Proposition~\ref{Singer-cycle-character-values}(i,ii).
We separate the computation into $s=1$ and $s \geq 2$,
and first compute $\sum_{U \in \CCC_1} f(s,U)$.  As $q=2$ there is only one
$U$ in $\CCC_1$, namely $U=\one$, and hence 
$$
\begin{aligned}
\sum_{U \in \CCC_1} f(s,U)=
f(1,\one) &=\sum_{k=0}^{n-1} 
       \deg(\chi^{\one,(n-k,1^k)})  \cdot
       \chi^{\one,(n-k,1^k)} (c^{-1}) \cdot
       \normchi^{\one,(n-k,1^k)}(z)^\ell \\
&=\sum_{k=0}^{n-1} 
     q^{\binom{k+1}{2}} \qbin{n-1}{k}{q} \cdot
     (-1)^k \cdot
     [n]^\ell_q(q^{n-k}-q^{n-k-1}-1)^\ell
\end{aligned}
$$
using the degree formula \eqref{hook-degree-formula} at $s=1$, the fact that
$
\chi^{(\one,n-k,1^k)} (c^{-1}) =
        (-1)^k \chi^{\one, (n)} (c^{-1}) = (-1)^k
$
from Proposition~\ref{Singer-cycle-character-values}(iii),
and the value $\normchi^{\one,(n-k,1^k)}(z_1) = [n]_q(q^{n-k}-q^{n-k-1}-1)$ from 
\eqref{transvection-class-value}.

For $s \geq 2$, we compute 
\begin{align*}
\sum_{\substack{(s,U):\\s |n,  s \geq 2 \\\ U \in \CCC_s}} f(s,U)
&=
\sum_{\substack{(s,U):\\s |n,  s \geq 2 \\\ U \in \CCC_s}}
   \sum_{k=0}^{ \frac{n}{s}-1 } 
     \deg(\chi^{U,\hook{k}{\frac{n}{s}}}) \cdot
      \chi^{U,\hook{k}{\frac{n}{s}}} (c^{-1}) \cdot
      \normchi^{U,\hook{k}{\frac{n}{s}}}(z_1)^\ell \\
&=\sum_{\substack{s |n\\ s \geq 2} }
   \sum_{k=0}^{ \frac{n}{s}-1 } 
          \frac{(-1)^{n-\frac{n}{s}} q^{s\binom{k+1}{2}} (q;q)_n}{(q^s;q^s)_{\frac{n}{s}}} 
            \qbin{\frac{n}{s}-1}{k}{q^s} \cdot
      \left( \sum_{U \in \CCC_s} \chi^{U,\hook{k}{\frac{n}{s}}}(c^{-1}) \right) \cdot
          (-[n]_q)^\ell 
\end{align*}
again via \eqref{hook-degree-formula},
Proposition~\ref{Singer-cycle-character-values}(iii),
and \eqref{transvection-class-value}. 
The parenthesized sum is $(-1)^{n-\frac{n}{s}-k}\mu(s)$ by 
Proposition~\ref{Singer-cycle-character-values}(iii, iv), so
\begin{align*}
\sum_{\substack{(s,U):\\s |n,  s \geq 2 \\\ U \in \CCC_s}} f(s,U)
&=(-[n]_q)^\ell (q;q)_n
 \sum_{\substack{s |n\\ s \geq 2} }
  \frac{1}{(q^s;q^s)_{\frac{n}{s}}}
   \left( \sum_{k=0}^{ \frac{n}{s}-1 } 
          (-1)^{k} q^{s\binom{k+1}{2}}
            \qbin{\frac{n}{s}-1}{k}{q^s} \right) \mu(s)\\
&= (-[n]_q)^\ell (q;q)_n
\sum_{\substack{s |n\\ s \geq 2} } 
     \frac{\mu(s)}{(q^s;q^s)_{\frac{n}{s}}}
              (q^s;q^s)_{\frac{n}{s}-1} \\
&= (-[n]_q)^\ell (q;q)_{n-1}
\sum_{\substack{s |n\\ s \geq 2} } \mu(s) \\ 
&= - (-[n]_q)^\ell (q;q)_{n-1},
\end{align*}
where the second equality used the $q$-binomial theorem \eqref{q-binomial-theorem}.  
Thus one has for $q=2$ that
\begin{equation}
\label{final-answer-for-q=2}
t_q(n,\ell)
=\frac{1}{|GL_n|} 
\left(
- (-[n]_q)^\ell (q;q)_{n-1}
+\sum_{k=0}^{n-1} 
     q^{\binom{k+1}{2}} \qbin{n-1}{k}{q} \cdot
     (-1)^k \cdot
     [n]^\ell_q(q^{n-k}-q^{n-k-1}-1)^\ell
\right).
\end{equation}
Since $|GL_n|=(-1)^n q^{\binom{n}{2}} (q;q)_n$, one finds
that \eqref{final-answer-for-q=2} agrees
with the expression \eqref{q-Jackson-sum-formula}  
$$
t_q(n,\ell)=\frac{(-[n]_q)^\ell}{q^{\binom{n}{2}}(q;q)_n} 
  \left(
    (-1)^{n-1} (q;q)_{n-1}+
    \sum_{k=0}^{n-1}(-1)^{k+n} q^{\binom{k+1}{2}} \qbin{n-1}{k}{q} 
                           (1+q^{n-k-1}-q^{n-k})^\ell
  \right)
$$
after redistributing the $[n]_q^\ell$ and powers of $-1$.  This completes the proof of
Theorem~\ref{q-factorization-theorem} for $q=2$.

\section{Further remarks and questions}
\label{questions-remarks} 

\subsection{Product formula versus partial fraction expansions}
The equivalence between \eqref{Jackson-ordinary-gf-product}, 
\eqref{Jackson-difference-formula}, and between
 \eqref{q-Jackson-ordinary-gf-product},
\eqref{q-Jackson-sum-formula} are explained as follows.
One checks the partial fraction expansion of
\eqref{Jackson-ordinary-gf-product} is
$$
\sum_{\ell \geq 0} t(n,\ell) x^\ell
\,\, = \,\,  
\dfrac{n^{n-2} x^{n-1}}
     {\prod_{k=0}^{n-1}  \left( 1 - x n\left(\frac{n-1}{2}-k\right) \right) }
=
        \frac{1}{n!} \sum_{k=0}^{n-1}
       \frac{(-1)^k \binom{n-1}{k}}
             {1-xn\left(\frac{n-1}{2}-k\right)} 
$$
and comparing coefficients of $x^\ell$ gives the first equality in
\eqref{Jackson-difference-formula}.

Similarly, one checks that the partial fraction expansion of
the right side of \eqref{q-Jackson-ordinary-gf-product} is
\begin{equation}
\label{q-partial-fraction-expansion}
\begin{aligned}
& (q^n-1)^{n-1} \cdot \dfrac{ x^n }
    {\left(1+x[n]_q \right) \prod_{k=0}^{n-1} \left(1+x[n]_q(1+q^k-q^{k+1})\right)} \\
&=\dfrac{(-1)^n}{q^{\binom{n}{2}} (q^n-1)}
   \left(
     \frac{1}{1+x[n]_q}
      + \sum_{k=0}^{n-1} \dfrac{(-1)^{k+1} 
                          q^{\binom{k+1}{2}}}{(q;q)_k (q;q)_{n-1-k}} \cdot
                              \dfrac{1}{1+x[n]_q(1+q^{n-k-1}-q^{n-k})}
   \right).
\end{aligned}
\end{equation}
Comparing coefficients of $x^\ell$ in \eqref{q-partial-fraction-expansion} 
gives \eqref{q-Jackson-sum-formula}. 
This proves \eqref{q-Jackson-ordinary-gf-product}.

\subsection{More observations about $t_q(n,\ell,m)$}

From \eqref{tot-num-formula} and \eqref{det-sequence-difference-formula}
one can derive $q=1$ limits
\[
\begin{array}{rll}
t_1(n,\ell)&:=
\lim_{ q\to 1} \frac{t_q(n,\ell)}{(1-q)^{n-1}}
&= (-n)^{\ell-1}\binom{\ell}{n}\\
t_1(n,\ell,m)&:=\lim_{q\to 1} t_q(n,\ell,m)&=n^{\ell-1}\binom{\ell-m-1}{\ell-n}.
\end{array}
\]
We do not know an interpretation for these limits.

\subsection{Exponential generating function}
The classical count $t(n,\ell)$ of
factorizations of an $n$-cycle into $\ell$ transpositions
has both an elegant ordinary generating 
function \eqref{Jackson-ordinary-gf-product} and 
{\it exponential} generating function
\begin{equation}
\label{Jackson-exponential-gf}
\sum_{\ell \geq 0} t(n,\ell) \frac{u^\ell}{\ell!}
=\frac{1}{n!} \left( e^{u\frac{n}{2}} - e^{-u\frac{n}{2}} \right)^{n-1}.
\end{equation}
This was generalized by Chapuy and Stump \cite{ChapuyStump}
to {\it well-generated} finite complex reflection group $W$ as
follows;  we refer to their paper for the background on such
groups.  If $W$ acts irreducibly on $\CC^n$, with a total
of $\Nref$ reflections and $\Nhyp$ reflecting hyperplanes, then
for any Coxeter element $c$, the number $a_\ell$ of ordered factorizations
$c=t_1 \cdots t_\ell$ into reflections satisfies
\begin{equation}
\label{well-generated-Coxeter-element-exp-gf}
\begin{aligned}
\sum_{\ell \geq 0} a_{\ell} \frac{u^\ell}{\ell!}
&=\frac{1}{|W|} \left( e^{u\frac{\Nref}{n}} - e^{-u\frac{\Nhyp}{n}} \right)^{n} \\
&=\frac{1}{|W|} e^{-u \Nhyp} \left( e^{u\frac{\Nref+\Nhyp}{n}} - 1 \right)^{n} \\
&=\frac{1}{|W|} e^{-u \Nhyp} \left[ \Delta^n \left( e^{ux\frac{\Nref+\Nhyp}{n}} \right)
\right]_{x=0}
\end{aligned}
\end{equation}
where the last equality uses the fact that the
difference operator $\Delta$
satisfies $\left[ \Delta^n(e^{ax})\right]_{x=0}=(e^{a}-1)^n$.  

One can derive an exponential generating function 
analogous to \eqref{well-generated-Coxeter-element-exp-gf} 
for the number $t_q(n,\ell)$ of Singer cycle factorizations in $W=GL_n(\FF_q),$
\begin{equation}
\label{tq-expgf}
\sum_{\ell \geq 0} t_q(n,\ell) 
 \frac{u^\ell}{\ell!}
=\frac{(q-1)^{n-1} q^{\binom{n}{2}}}{|W|} 
\,\, e^{-u \Nhyp} \left[ \Delta_q^{n-1} \left( 
                    \frac{1}{x} \left( e^{ux\frac{\Nref+\Nhyp}{q^{n-1}}} - 1
                \right) \right) \right]_{x=1},
\end{equation}
where $\Nhyp, \Nref$ denote the number of reflecting hyperplanes and
reflections in $W=GL_n(\FF_q)$, that is, 
$$
\begin{aligned}
\Nhyp&=[n]_q,\\
\Nref&=[n]_q(q^n-q^{n-1}-1).
\end{aligned}
$$

To prove \eqref{tq-expgf}, use \eqref{q-Jackson-difference-formula} to find
$$
\begin{aligned}
\sum_{\ell \geq 0} t_q(n,\ell) 
 \frac{u^\ell}{\ell!} =&\frac{(1-q)^{n-1}}{(q;q)_n}
     \left[ \Delta_q^{n-1}\biggl( 
           \frac{1}{x}\bigl( e^{-u[n]_q}-e^{-u[n]_q(1+x(1-q))}\bigr)\biggr) 
     \right]_{x=1}
 \\
=& \frac{(-1)^n (1-q)^{n-1} q^{\binom{n}{2}}}{|W|} e^{-u[n]_q}
  \left[ \Delta_q^{n-1}\biggl( 
    \frac{1}{x}\bigl( 1-e^{u x [n]_q (q-1))}\bigr)\biggr) 
  \right]_{x=1}.
\end{aligned}
$$
Noting that $[n]_q=\Nhyp$, and $[n]_q(q-1)=q^n-1=(\Nhyp+\Nref)/q^{n-1}$,
and distributing some negative signs, gives  \eqref{tq-expgf}.

\subsection{Hurwitz orbits}
\label{Hurwitz-orbit-conjectures-section}

In a different direction, one can consider
the {\it Hurwitz action} of the {\it braid group on $\ell$ strands}
acting on length $\ell$ ordered factorizations $c=t_1 t_2 \cdots t_\ell$.
Here the braid group generator $\sigma_i$ acts on ordered factorizations as follows:
$$
\begin{array}{rccll}
(t_1,\ldots,t_{i-1},& t_i,&t_{i+1},& t_{i+2},\ldots,t_\ell) &\overset{\sigma_i}{\longmapsto} \\
(t_1,\ldots,t_{i-1},& t_{i+1},& t_{i+1}^{-1} t_i t_{i+1},& t_{i+2},\ldots,t_\ell).
\end{array}
$$
For well-generated complex reflection groups $W$ of rank $n$ and taking $\ell=n$, 
Bessis showed \cite[Prop. 7.5]{Bessis-Kpi1} 
that the set of all shortest ordered factorizations $(t_1,\ldots,t_n)$ of
a Coxeter element $c=t_1 t_2 \cdots t_n$ forms
a single transitive orbit for this Hurwitz action.

One obvious obstruction to an analogous transitivity assertion for $c$ a Singer cycle in 
$GL_n(\FF_q)$ and factorizations $c=t_1 t_2 \cdots t_\ell$ is that 
the unordered $\ell$-element multiset $\{\det(t_i)\}_{i=1}^\ell$ of $\FF_q^\times$
is constant on a Hurwitz orbit, but (when $q \neq 2$) can vary between different factorizations,
even when $\ell=n$.  Nevertheless, we make the following 
conjecture.

\begin{conjecture}
\label{Hurwitz-orbit-conjecture}
Any two factorizations $c=t_1 t_2 \cdots t_\ell$ with
the same multiset $\{\det(t_i)\}_{i=1}^\ell$ lie in the same Hurwitz orbit.
In particular, there is only one Hurwitz orbit of factorizations when $q=2$ for 
any $\ell$.
\end{conjecture}

We report here some partial evidence for Conjecture~\ref{Hurwitz-orbit-conjecture}.

\begin{itemize}
\item  It is true when $n=\ell=2$;  here is a proof.  
Fix a Singer cycle $c$ in $GL_2(\FF_q)$ and $\alpha_1, \alpha_2$ in $\FF_q^\times$
having $\det(c)= \alpha_1 \alpha_2$. Theorem~\ref{fixed-det-sequence-theorem} 
in the case $\ell = n = 2$ tells us that there will be exactly $[2]_q=q + 1$
factorizations $c = t_1 \cdot t_2$ of $c$ as a product of two reflections
with $(\det(t_1),\det(t_2))=(\alpha_1,\alpha_2)$, and similarly 
$q + 1$ for which $(\det(t_1),\det(t_2))=(\alpha_2,\alpha_1)$.  This gives
a total of  either $q+1$ or $2(q+1)$ factorizations with this multiset of determinants,
depending upon whether or not $\alpha_1 = \alpha_2$.
Now note that applying the Hurwitz action $\sigma_1$ twice sends
\[
(t_1, t_2) \overset{\sigma_1}{\longmapsto} 
(t_2 \,\, , \,\, t_2^{-1} t_1 t_2) \overset{\sigma_1}{\longmapsto} 
(t_2^{-1} t_1 t_2, \,\, \underbrace{t_2^{-1} t_1^{-1} t_2 t_1 t_2}_{=c^{-1} t_2 c}),
\]
yielding a factorization with the same determinant sequence,
but whose second factor changes from $t_2$ to $c^{-1} t_2 c$.
This moves the reflecting hyperplane (line) $H$ for $t_2$ to the line
$c^{-1}H$ for $c^{-1} t_2 c$.  Since 
$\FF_{q^2}^\times=\langle c \rangle$, one knows that the powers of 
$c$ act transitively on the lines in $\FF_{q^2} \cong \FF_q^2$, and hence
there will be at least $q+1$ different second factors 
$\{ c^{-i} t_2 c^i \}$ achieved in the Hurwitz orbit.  This shows that
the Hurwitz orbit contains {\it at least} $q+1$ or $2(q+1)$ different factorizations,
depending upon whether or not $\alpha_1 = \alpha_2$, so it
exhausts the factorizations that achieve this multiset of determinants.
This completes the proof.

\item Conjecture~\ref{Hurwitz-orbit-conjecture}
has also been checked 
\begin{itemize} 
\item for $q=2$ when $n=\ell \leq 5$
and $n=3,\ell=4$, 
\item for $q=3$ when $n=2$ and $\ell \leq 4$,
and also when $n=\ell=3$,
\item 
for $q=5$ when $n=2$ and $\ell \leq 3$. 
\end{itemize}
\end{itemize}

One might hope to prove Conjecture~\ref{Hurwitz-orbit-conjecture}
similarly to the uniform proof for transitivity of the Hurwitz action on
short reflection factorizations of Coxeter elements in real reflection groups,
given in earlier work of Bessis \cite[Prop. 1.6.1]{Bessis}.  
His proof is via induction on the rank, and 
relies crucially on proving these facts:
\begin{itemize}
\item
The elements $w \leq c$ in the {\it absolute order}, that is, the elements which appear as 
partial products $w=t_1 t_2 \cdots t_i$ in shortest factorizations $c=t_1 t_2 \cdots t_n$, 
are all themselves {\it parabolic Coxeter elements}, that is, Coxeter elements
for conjugates of standard parabolic subgroups of $W$.
\item All such parabolic Coxeter elements share the property that the Hurwitz action is transitive
on their shortest factorizations into reflections.
\end{itemize}

One encounters difficulties in trying to prove this analogously,
when one examines the interval $[e,c]$ of elements lying below a Singer
cycle $c$ in $GL_n(\FF_q)$:
\begin{itemize}
\item
It is no longer true that the elements $g$ in $[e,c]$
all  have a transitive Hurwitz action on their own short factorizations.  
For example in $GL_4(\FF_2)$, the unipotent element $u$ equal to a single Jordan block of size
$4$
appears as a partial product on the way to factoring a Singer cycle, but its $64$ short
factorizations $u=t_1 t_2 t_3$ into reflections break up into two Hurwitz orbits, 
of sizes $16$ and $48$.
\item
It also seems nontrivial to characterize intrinsically the elements in $[e,c]$
for a fixed Singer cycle $c$.  For example, the elements $g$ which 
are {\it $c$-noncrossing} in the following sense appear\footnote{That is, it is true for
$GL_n(\FF_2)$ with $n=2,3,4$ and
also for $GL_n(\FF_3)$ with $n=2,3$.} to be always among them:  
arranging the elements $\FF_{q^n}^\times=\{1,c,c^2,\ldots,c^{q^n-2}\}$ clockwise circularly, 
$g$ permutes them (after embedding them via $\FF_{q^n} \cong \FF_q^n$) in cycles that are
each
oriented clockwise, and these oriented arcs do not cross each other.  However, starting already
with 
$GL_4(\FF_2)$ and $GL_3(\FF_3)$, there are {\it other} element 
below the Singer cycle besides these $c$-noncrossings.  
\end{itemize}

%

%

\subsection{$q$-Noncrossings?}
\label{q-noncrossings-section}

The poset of elements $[e,c]$ lying below a Singer cycle $c$ in the absolute order
on $GL_n(\FF_q)$ would seem like a reasonable candidate for a $q$-analogue of the
usual poset of {\it noncrossing partitions} of $\{1,2,\ldots,n\}$; see \cite{Armstrong}.
However, $[e,c]$ does not seem to be so well-behaved in $GL_n(\FF_q)$, 
although a few things were proven about it by Jia Huang in \cite{Huang}.  

For instance, he showed that the absolute length
of an element $g$ in $GL_n(\FF_q)$, that is, the minimum length of a factorization into
reflections, coincides with the codimension of the fixed
space $(\FF_q^n)^g$.  Hence the poset $[e,c]$ is ranked in a similar
fashion to the noncrossing partitions of real reflection groups,
and has an order- and rank-preserving map 
\[
\begin{array}{rcl}
[e,c] &\overset{\pi}{\longrightarrow}&L(\FF_q^n) \\
g &\longmapsto & (\FF_q^n)^g
\end{array}
\]
to the lattice $L(\FF_q^n)$ of subspaces of $\FF_q^n$.  Because conjugation by $c$ acts
transitively on lines and hyperplanes, this map is surjective for $n \leq 3$; empirically, it seems
to be surjective in general.  The poset $[e,c]$ also has a {\it Kreweras complementation}
anti-automorphism $w \mapsto w^{-1}c$.

However, Huang noted that the rank sizes of $[e,c]$
do not seem so suggestive. E.g., for $[e,c]$ in $GL_4(\FF_2)$ they
are $(1,60,240,60,1)$, and preclude $\pi$ being an $N$-to-one map for some integer $N$,
since $L(\FF_2^4)$ has rank sizes $(1,15,35,15,1)$ and $35$ does not divide $240$.

\begin{question}
Are the {\it $c$-noncrossing elements} mentioned in
Section~\ref{Hurwitz-orbit-conjectures-section}
a better-behaved subposet of $[e,c]$?  
\end{question}

\subsection{Regular elliptic elements versus Singer cycles}
\label{regular-elliptic-elements-remark}
Empirical evidence supports the following
hypothesis regarding the {\it regular elliptic elements} of $GL_n(\FF_q)$
that appeared in Proposition~\ref{regular-elliptic-definition-proposition}.

\begin{conjecture}
\label{reg-ell-conj}
The number of ordered reflection factorizations $g=t_1 t_2 \cdots t_\ell$ 
is the same for all regular elliptic elements $g$ in $GL_n(\FF_q)$,
namely the quantity $t_q(n,\ell)$ that appears 
in Theorem~\ref{q-factorization-theorem}.
\end{conjecture}
\noindent
Conjecture~\ref{reg-ell-conj} has been verified for $n=2$ and $n=3$ 
using explicit character values \cite{Steinberg}. In the case $\det g\neq 1$, 
only minor modifications are required in our arguments to prove
Conjecture~\ref{reg-ell-conj}.  
The spot in our proof  that breaks down for regular elliptic elements 
with $\det g= 1$
is the identity \eqref{moebius-function-equality}.  For example,   
when $s=n=4$ and $q=2$, if one
chooses $\beta$ in $\FF_{2^4}^\times$ with $\beta^5=1$ (so still one
has $\FF_{2^4}=\FF_2(\beta)$, but $\FF_{2^4}^\times \neq \langle \beta \rangle$),
then there are three homomorphisms $\varphi$ with
free Frobenius orbits and 
$
\sum_{\phi} \left( 
        \varphi(\beta) + \varphi(\beta^2) + \varphi(\beta^4) + \varphi(\beta^8) 
    \right)=-3 \quad (\neq 0 = \mu(4)).
$
Nevertheless, in this $GL_4(\FF_2)$ example it appeared from {\tt GAP} \cite{GAP} 
computations that such regular elliptic $g$ with $g^5=1$ had the same number of factorizations 
into $\ell$ reflections for all $\ell$ as did a Singer cycle in $GL_4(\FF_2)$.

\begin{remark}
On the other hand, in considering transitivity of Hurwitz actions, we {\it did}
see a difference in behavior for regular elliptic elements versus Singer cycles:  
in $GL_4(\FF_2)$, there are $3375=(2^4-1)^{4-1}$ short reflection
factorizations $t_1 t_2 t_3 t_4$ both for the
the Singer cycles (the elements whose characteristic polynomials are
$x^4+x^3+1$ or $x^4+x+1$) and for the non-Singer cycle regular elliptic elements
(the elements whose characteristic polynomials are $x^4+x^3+x^2+x+1$).
However, for the Singer cycles, these factorizations form one Hurwitz orbit,
while for the non-Singer cycle regular elliptic elements
they form four Hurwitz orbits.
\end{remark}

\subsection{The approach of Hausel, Letellier, and Rodriguez-Villegas}
The number of factorizations  $g=t_1 t_2 \cdots t_\ell$ where  $t_1,\ldots,t_\ell,g$ come from specified $GL_n(\FF_q)$ conjugacy
classes $C_1,\ldots,C_\ell,C_{\ell+1}$
appears in work of Hausel, Letellier, and Rodriguez-Villegas \cite{HauselLetellierRodriguez} and more
recently Letellier \cite{Letellier}.  They interpret it in terms of the topology of objects called
{\it character varieties}  under certain {\it genericity conditions} \cite[Definition 3.1]{Letellier} on the conjugacy classes. 
One can check that these conditions are satisfied in the case of interest to us, that is,  when $C_{\ell+1}$ is a conjugacy class of Singer cycles and the $C_1,\ldots,C_\ell$ are all conjugacy classes of reflections.   Assuming these genericity conditions, \cite[Theorem 4.14]{Letellier} gives an expression for the number of such factorizations in terms of a specialization 
$\HH_\omega(q^{-\frac{1}{2}},q^{\frac{1}{2}})$ of a rational function $\HH_\omega(z,w)$ defined 
in \cite[\S 1.1]{HauselLetellierRodriguez} via {\it Macdonald symmetric functions}.
In principle, this expression should recover Theorem~\ref{fixed-det-sequence-theorem} as a very special case.  However, in practice, the calculation of $\HH_\omega(z,w)$ is sufficiently intricate that we have not verified it.

\subsection{Jucys-Murphy approach?}

The formulas for character values on semisimple reflections and transvections 
in Lemma~\ref{normalized-characters-on-semisimple-reflections}(ii)
and Corollary~\ref{normalized-characters-on-transvections}
are remarkably simple compared to the machinery used in their proofs.
Can they be developed using a $q$-analogue
of the Okounkov-Vershik approach \cite{CST, VershikOkounkov} to the ordinary character
theory 
of $\Symm_n$, using the commuting family of {\it Jucys-Murphy elements} 
\cite{Jucys, Murphy}, 
a multiplicity-free branching rule, a Gelfand-Zetlin basis, etc.?  
Such a theory might even allow one to prove $q$-analogues for more
general generating function results, such as one finds in Jackson \cite{Jackson}.

A feature of the $\Symm_n$ theory (see Chapuy and Stump \cite[\S 5]{ChapuyStump}, Jucys
\cite[\S 4]{Jucys}) 
is that any symmetric function $f(x_1,\ldots,x_n)$ when evaluated on the 
Jucys-Murphy elements $J_1,\ldots,J_n$ acts as a scalar 
in each $\Symm_n$-irreducible $V^\lambda$,
and this scalar is $f(c(a_1),\ldots,c(a_n))$ where $c(a_i)$ are the contents
of the cells of $\lambda$.  Taking $f=\sum_{i=1}^n x_i$ gives a 
quick calculation of the irreducible characters evaluated on $\sum_{i=1}^n J_i$, the sum of all
transpositions.  
Lemma~\ref{normalized-characters-on-semisimple-reflections}(ii)
and Corollary~\ref{normalized-characters-on-transvections} seem suggestive of 
a $q$-analogue for this assertion. 

It is at least clear how one might define relevant
Jucys-Murphy elements.

\begin{definition}
For $1 \leq m < n$ embed $GL_m \subset GL_n$ as 
the subgroup fixing $e_{m+1},\ldots,e_n$.  Then 
for each $\alpha \in \FF_q^\times$, let $J_m^{\alpha}:=\sum_t t$ be the sum inside
the group algebra $\CC GL_n$ over this subset of reflections:
\begin{equation}
\label{fine-Jucys-Murphy-summation-set}
\{ \text{reflections }t \in GL_{m}: \det(t)=\alpha\text{ and }t \not\in GL_{m-1}\}.
\end{equation}
\end{definition}

\begin{proposition}
The elements $\{ J_m^{\alpha} \}$ for $m=1,2,\cdots,n$ and  $\alpha $ in $\FF_q^\times$
pairwise commute.
\end{proposition}
\begin{proof}
Note that $J_n^{\alpha}$ commutes with any $g$ in $GL_{n-1}$,
or equivalently, $gJ_n^{\alpha}g^{-1}=J_n^{\alpha}$, since conjugation by $g$
induces a permutation of the set in \eqref{fine-Jucys-Murphy-summation-set}.
This shows that $J_n^{\alpha}, J_m^{\beta}$ commute
when $n \neq m$, since if one assumes $m<n$, then every term of 
$J_m^\beta$ lies in $GL_{n-1}$.  To see that $[J_n^{\alpha}, J_n^{\beta}]=0$,
note that our conjugacy sums $z_\alpha=:z_{n,\alpha}$ from 
Definition~\ref{reflection-conjugacy-class-sum-definition} lie the center of $\CC GL_n$
and can be expressed as 
$
z_{n,\alpha}=\sum_{i=1}^n J_i^{\alpha}.
$
Therefore 
$$
0=[z_{n,\alpha},J_n^\beta]
=\left[ \sum_{i=1}^n J_i^{\alpha}, J_n^{\beta} \right]
=[J_n^{\alpha}, J_n^{\beta}] +
   \left[ \sum_{i<n} J_i^{\alpha}, J_n^{\beta} \right] 
 = [J_n^{\alpha}, J_n^{\beta}]
$$ 
using bilinearity of commutators, and the commutativity of $J_i^{\alpha}, J_n^{\beta}$
for $i <n$.
\end{proof}

\section*{Acknowledgements}
The authors thank A. Ram and P. Diaconis for pointing them to this work of
Hildebrand \cite{Hildebrand} used in Section~\ref{transvection-character-section}.  They also thank A. Henderson and E. Letellier for pointing them to \cite{HauselLetellierRodriguez, Letellier}.

\end{document}